\documentclass[journal]{IEEEtran}

\usepackage[ansinew]{inputenc}
\usepackage{amsmath, amsfonts}
\usepackage{amsthm, amssymb}
\usepackage[english]{babel}
\usepackage[T1]{fontenc}
\usepackage{lmodern}
\usepackage{graphicx}
\usepackage{subfigure}
\usepackage{multirow}

\newtheorem{theorem}{Theorem}

\theoremstyle{definition}

\newtheorem{remark}{Remark}

\usepackage{verbatim}
\usepackage{array}
\usepackage{float}
\usepackage{subfigure}
\usepackage{algorithm, algorithmicx, algpseudocode}
\usepackage{color}

\newcommand{\guillemets}[1]{``#1''}
\newcommand{\ve}[1]{\mathbf{#1}}

\newcommand{\blue}[1]{\textcolor{black}{#1}}
\newcommand{\mat}[2]{\left(\begin{array}{#1} #2 \end{array} \right)}

\title{Koopman-based lifting techniques for nonlinear systems identification}

\author{A. Mauroy and J. Goncalves
\thanks{A. Mauroy is with Department of Mathematics and Namur Center for Complex Systems (naXys), University of Namur, Belgium (email: alexandre.mauroy@unamur.be). J. Gocalves is with the Luxembourg Centre for Systems Biomedicine, University of Luxembourg, Belvaux, Luxembourg (email: jorge.goncalves@uni.lu).}}

\begin{document}

\maketitle

\begin{abstract}

We develop a novel lifting technique for nonlinear system identification based on the framework of the Koopman operator. The key idea is to identify the linear (infinite-dimensional) Koopman operator in the lifted space of observables, instead of identifying the nonlinear system in the state space, a process which results in a linear method for nonlinear systems identification. The proposed lifting technique is an indirect method that does not require to compute time derivatives and is therefore well-suited to low-sampling rate datasets.

Considering different finite-dimensional subspaces to approximate and identify the Koopman operator, we propose two numerical schemes: a main method and a dual method. The main method is a parametric identification technique that can accurately reconstruct the vector field of a broad class of systems. The dual method provides estimates of the vector field at the data points and is well-suited to identify high-dimensional systems with small datasets. The present paper describes the two methods, provides theoretical convergence results, and illustrates the lifting techniques with several examples.

\end{abstract}

\section{Introduction}
\label{sec:intro}

The problem of identifying \blue{governing equations of continuous-time dynamical systems} from time-series data has attracted considerable interest in many fields such as biology, finance, and engineering. It is also closely related to network inference, which aims at reconstructing the interactions between the different states of a system, a problem of paramount importance in systems biology. In many cases, the identification problem is challenging due to the nonlinear nature of the systems and must be tackled with black-box methods (e.g. Wiener and Volterra series models \cite{nonlin_identif_Wiener}, nonlinear auto-regressive models \cite{nonlin_identif_narx}, neural network models \cite{nonlin_identif_neural}, see also \cite{nonlin_identif_survey1,nonlin_identif_survey2} for a survey). \blue{These methods are related to the classic approach to system identification \cite{Ljung}: they typically deal with long, highly-sampled time-series and provide a relationship between the system inputs and outputs.}

In the related context of nonlinear parameter estimation, a \blue{large body of} methods have been developed to identify the state dynamics of autonomous systems with a known structure \blue{(see e.g. \cite{nonlin_estim_book,Aastrom_identif} and references therein).} Typical methods seek the best linear combination of time derivatives of the state over a set of library functions (similar to the basis functions used in black-box models) \cite{param_estim_der_estim}. Similar approaches have also been proposed recently, partly motivated by the network identification problem (e.g. Bayesian approach \cite{Wei}, SINDy algorithm \cite{Brunton}). The above-mentioned methods are \emph{direct methods}, and their main advantage is that they rely on static linear regression techniques. However, they assume that time derivatives of the state can be accurately estimated (e.g. by using collocation techniques), a requirement that becomes prohibitive when the sampling time is too low, the measurements too noisy, or the time-series too short \blue{(e.g. biology)}. Instead, \emph{indirect} methods solve an initial value problem and \blue{do not require the estimation of} time derivatives \cite{param_estim_nonlin_lsq2}. Hence, they offer a good alternative to direct methods, 
but at the expense of solving a (nonconvex) nonlinear least squares problem. The goal of this paper is to propose a new indirect method for \blue{estimating the state dynamics (i.e. governing equations) of nonlinear dynamical systems. In the context of nonlinear system identification/parameter estimation, this method} not only circumvents the estimation of time derivatives but also relies on linear least squares optimization.

The approach proposed in this paper is based on the framework of the so-called Koopman operator \cite{Budisic_Koopman,Koopman}. The Koopman operator is a linear infinite-dimensional operator that describes the evolution of observable-functions along the trajectories of the system. Starting with the seminal work of \cite{Mezic}, several studies have investigated the interplay between the spectral properties of the operator and the properties of the associated system, a body of work that has led to new methods for the analysis of nonlinear systems (e.g. global stability analysis \cite{Mauroy_Mezic_stability}, global linearization \cite{Lan}, monotone systems \cite{Sootla_Mauroy_basins}, delayed systems \cite{Muller_delay_Koopman}). While the above-mentioned studies focus on systems described by a known vector field, the Koopman operator approach is also conducive to data analysis and directly connected to numerical schemes such as Dynamic Mode Decomposition (DMD) \cite{Arbabi,Rowley,Schmid,Tu}. This yielded another set of techniques for data-driven analysis and control of nonlinear systems (observer synthesis \cite{Surana}, model predictive control \cite{Korda_MPC}, optimal control \cite{Kaiser_eigenfunction_control}, power systems stability analysis \cite{Susuki}, to list a few). In this context, this paper aims at connecting data to vector field, thereby bridging these two sets of methods.

The Koopman operator provides a linear representation of the nonlinear system in a lifted (infinite-dimensional) space of observable-functions. \blue{Through this lifting approach, one can therefore \emph{identify the linear Koopman operator in the space of observables} (see e.g. \cite{Rowley_EDMD}), instead of identifying the nonlinear system in the state space. Our numerical scheme exploits this idea and} proceeds in three steps: (1) lifting of the data, (2) identification of the Koopman operator, and (3) identification of the vector field. In the first step, snapshot data are lifted to the space of observables. In the second step, we derive two distinct methods: (a) a main method which identifies a representation of the Koopman operator in a basis of functions; (b) a dual method which identifies the representation of the operator in the \guillemets{sample space}. In the third step, we connect the vector field to the infinitesimal generator of the identified operator and solve a linear least squares problem to compute the linear combination of the vector field in a basis of library functions. The two methods are complemented with convergence results showing that they identify the vector field exactly in optimal conditions. The main method has been initially proposed in \cite{Mauroy_Goncalves_CDC} and a similar approach developed in a stochastic framework \blue{and based on non-convex optimization} can also be found in the more recent work \cite{Riseth}. \blue{It should be noted that the lifting technique is not new. More precisely, the first steps of our methods (i.e. lifting and identification of the operator) are directly related to a component of the Extended Dynamic Mode Decomposition (EDMD) technique \cite{Rowley_EDMD} (main method) or inspired from kernel-based EDMD technique \cite{Williams_kernel} (dual method). Although EDMD techniques focus on the spectral properties of the operator, their lifting approach could also be used for prediction \cite{Korda_MPC}. In contrast, the main goal of the two methods proposed in this paper is not to predict trajectories, but to provide a functional representation of the vector field. This representation can be further used for system analysis (e.g. existence of equilibria, stability) and model-based control, and is also directly related to the network identification problem. Note however that the main method has recently been used with success in the context of robot motion prediction \cite{lifting_robot}.
}

The proposed lifting technique has several advantages. First of all, it relies only on linear methods which are easy and efficient to implement. It is also well-suited to data acquired from short time-series with low sampling rates \blue{(e.g. several experiments in biology, with a few costly measurements).} Although initially limited to polynomial vector fields, the main method works efficiently with a broad class of behaviors, including unstable and chaotic systems. In addition, the dual method is well-suited to identify large-dimensional systems and to reconstruct network topologies, in particular when the number of sample points is smaller than the unknown system parameters. Finally, lifting techniques can be extended to identify non-polynomial vector fields and open systems (with input or process noise). \blue{In contrast to these advantages, a main limitation of the methods is that they require full state measurements and therefore cannot provide an input-output representation of the system.}

The rest of the paper is organized as follows. In Section \ref{sec:Koopman}, we present the problem and introduce the general lifting technique used for system identification. Section \ref{sec:lifting} describes the main method and provides theoretical convergence results, while Section \ref{sec:extensions} discusses some extensions of the methods to non-polynomial vector fields and open systems. In Section \ref{sec:dual}, we propose the dual method to identify high-dimensional systems with small datasets and give convergence proofs. The two methods are illustrated with several examples in Section \ref{sec:examples}, where the network reconstruction problem is also considered. Concluding remarks and perspectives are given in Section \ref{sec:conclu}.

\section{Identification in the Koopman operator framework}
\label{sec:Koopman}

\subsection{Problem statement}

We address the problem of identifying the vector field of a nonlinear system from time series generated by its dynamics. We consider the system
\begin{equation}
\label{syst1}
\dot{\ve{x}} = \ve{F}(\ve{x}) \,, \quad \ve{x} \in \mathbb{R}^n
\end{equation}
where the vector field $\ve{F}(\ve{x})$ is of the form
\begin{equation}
\label{eq:vec_field}
\ve{F}(\ve{x}) = \sum_{k=1}^{N_F} \ve{w}_k \, h_k(x) \,.
\end{equation}
The vectors $\ve{w}_k=(w_k^1 \, \cdots \, w_k^n)^T \in \mathbb{R}^n$ are unknown coefficients (to be identified) and the library functions $h_k$ are assumed to be known. Note that some coefficients might be equal zero. Unless stated otherwise, we will consider that the vector field is polynomial, so that $h_k$ are monomials: $h_k=p_k$ with
\begin{equation}
\label{eq:monomials}
p_k(\ve{x}) \in \{ x_1^{s_1} \cdots x_n^{s_n} | (s_1,\dots,s_n) \in \mathbb{N}^n, s_1+\cdots+s_n \leq  m_F \}
\end{equation}
where $m_F$ is the total degree of the polynomial vector field. The number of monomials in the sum \eqref{eq:vec_field} is given by $N_F=(m_F+n)!/(m_F! n!)$.
As shown in Section \ref{sec:other_vec}, the proposed method can also be generalized to other types of vector fields in a straightforward way.

Our goal is to identify the vector field $\ve{F}$ (i.e. the $N_F$ coefficients $\ve{w}_k$) from snapshot measurements of the system trajectories. We consider $K$ snapshot pairs $(\ve{x}_k,\ve{y}_k)$ obtained from noisy measurements (proportional to the exact state value): we have
\begin{equation}
\label{eq:meas_noise}
\ve{x}_k = \ve{\bar{x}}_k + \epsilon(\ve{x}_k) \qquad \ve{y}_k = \ve{\bar{y}}_k + \epsilon(\ve{y}_k)
\end{equation}
where $\epsilon$ is the state-dependent measurement noise, and
\begin{equation}
\label{rel_x_y}
\ve{\bar{y}}_k=\varphi^{T_s}(\ve{\bar{x}}_k)
\end{equation}
where $t \mapsto \varphi^t(\ve{x_0})$ is the solution to \eqref{syst1} associated with the initial condition $\ve{x_0}$. We assume that the measurement noise is Gaussian and proportional to the state value, i.e. $\epsilon(\ve{x})=\ve{x} \odot \ve{v}$ where $\odot$ denotes the element-wise product and $\ve{v}$ is a Gaussian random variable with zero mean and standard deviation $\sigma_{meas}$. We also assume that all pairs $(\ve{x}_k,\ve{y}_k)$ lie in a compact set $X \subset \mathbb{R}^n$ and are obtained with the same sampling period $T_s$. They can belong to a single trajectory or to multiple trajectories. Stochastic systems with process noise and systems with inputs will also be considered (see Section \ref{sec:extensions}).

\begin{remark}
	For numerical reasons, we will assume in general that the data points lie in a set $X \subset [-1,1]^n$. If original data do not satisfy this assumption, then they can be rescaled to yield new data pairs $(\ve{x}'_k,\ve{y}'_k) = (\ve{x}'_k/\alpha,\ve{y}'_k/\alpha) \in [-1,1]^{2n}$. These new pairs enable to identify a vector field $\ve{F}'(\ve{x})$ with coefficients $\ve{w}'_k = \alpha^{m_k-1} \ve{w}_k$, where $m_k$ is the total degree of the monomial $p_k$. \hfill $\diamond$
\end{remark}

\subsection{Koopman operator}

System \eqref{syst1} represents the state dynamics in $\mathbb{R}^n$. Alternatively, the system can be described in a lifted space $\mathcal{F}$ of observable-functions $f:\mathbb{R}^n \to \mathbb{R}$. Provided that the observable functions are continuously differentiable, their dynamics in the lifted space are given by
\begin{equation}
\label{syst_infinite}
\dot{f} = (\ve{F} \cdot \nabla) f \,,  \quad f \in \mathcal{F},
\end{equation}
where $\dot{f}$ denotes $\partial (f \circ \varphi^t)/\partial t$ (with a slight abuse of notation) and $\nabla$ denotes the gradient (see e.g. \cite{Lasota_book}). In contrast to \eqref{syst1}, the dynamics \eqref{syst_infinite} are infinite-dimensional but linear. 

While the flow induced by \eqref{syst1} in the state space is given by the nonlinear flow map $\varphi$, the flow induced by \eqref{syst_infinite} in the lifted space is given by the linear semigroup of Koopman operators $U^t:\mathcal{F} \to \mathcal{F}$, $t\geq 0$. This semigroup governs the evolution of the observables along the trajectories, i.e.
\begin{equation*}
U^t f = f \circ \varphi^t \,.
\end{equation*}
Under appropriate conditions  (see Section \ref{sec:proofs}), the semigroup of Koopman operators is strongly continuous and generated by the operator
\begin{equation}
\label{inf_gen}
L = \ve{F} \cdot \nabla
\end{equation}
appearing in \eqref{syst_infinite}. In this case, we use the notation
\begin{equation}
\label{Ut_L}
U^t = e^{L t } \,.
\end{equation}
The operator $L$ is called the infinitesimal generator of the Koopman operator and we denote its domain by $\mathcal{D}(L)$. 

\subsection{Linear identification in the lifted space}

There is a one-to-one correspondence between systems of the form \eqref{syst1} and lifted systems \eqref{syst_infinite}, or equivalently between the flow $\varphi^t$ and the semigroup of Koopman operators $U^t$. Exploiting this equivalence, we propose to solve the identification problem in the lifted space instead of the state space. This can be done in three steps (see Figure \ref{Koopman_ident}).
\begin{enumerate}
\item{\textbf{Lifting of the data.}} Snapshots pairs $(\ve{x}_k,\ve{y}_k)$ are lifted to the space of observable by constructing new pairs of the form $(g(\ve{x}_k),g(\ve{y}_k))$ for some $g \in \mathcal{F}$. The functions $g$ are assumed to be continuously differentiable and we call them \emph{basis functions}. It follows from \eqref{eq:meas_noise} and \eqref{rel_x_y} that
\begin{equation}
\label{f_y}
g(\ve{y}_k) = g(\varphi^{T_s}(\ve{x}_k-\epsilon(\ve{x}_k))+\epsilon(\ve{y}_k)) \approx U^{T_s} g(\ve{x}_k) + \mathcal{O}(\| \epsilon \|)\,.
\end{equation}

\item{\textbf{Identification of the Koopman operator.}} A finite-dimensional projection of the Koopman operator is obtained through a classic linear identification method that is similar to a component of the Extended Dynamic Mode Decomposition (EDMD) algorithm \cite{Rowley_EDMD}. This yields (an approximation of) the infinitesimal generator $L$ of the Koopman operator. 
\item{\textbf{Identification of the vector field.}} Using \eqref{inf_gen}, we can finally obtain the vector field $\ve{F}$.
\end{enumerate}

\begin{figure}[h]
   \centering
    \includegraphics[width=6cm]{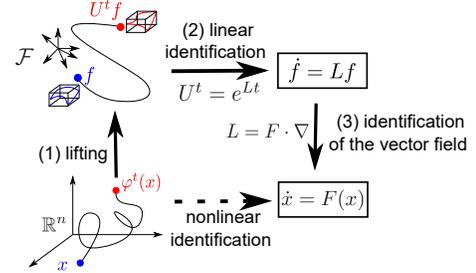}
   \caption{Classical nonlinear system identification is performed directly in the state space. In contrast, the proposed lifting technique consists of three steps: (1) lifting of the data; (2) linear identification of the Koopman operator in the lifted space; (3) identification of the vector field.}
   \label{Koopman_ident}
\end{figure}

\section{The main lifting method}
\label{sec:lifting}

\subsection{Description of the method}

This section describes in detail the three steps of our main method. The first step and the first part of the second step are related to a component of the EDMD algorithm (see \cite{Rowley_EDMD} for more details).

\subsubsection{First step - lifting of the data}
\label{sec:lifting_data}

The data must be lifted to the infinite-dimensional space $\mathcal{F}$ of observables. However, the method has to be numerically tractable and is developed in a finite-dimensional linear subspace $\mathcal{F}_N \subset \mathcal{F}$ spanned by a basis of $N$ linearly independent functions. The choice of basis functions $\{g_k\}_{k=1}^N$ can be arbitrary (e.g. Fourier basis, radial basis functions), but might affect the method performances. 
Since the vector field is assumed to be polynomial, we naturally choose the basis of monomials $\{g_k\}_{k=1}^N=\{p_k\}_{k=1}^N$ with total degree less or equal to $m$ to facilitate the representation of the Koopman operator. The number of basis functions is equal to $N=(n+m)!/(n! \, m!)$. \blue{We impose $m \geq m_F$.}

For each snapshot pair $(\ve{x}_k,\ve{y}_k) \in \mathbb{R}^{n \times 2}$, $k \in \{1,\dots,K\}$, we construct a new pair $(\ve{p}(\ve{x}_k),\ve{p}(\ve{y}_k)) \in \mathbb{R}^{N \times 2}$, where $\ve{p}(\ve{x}) = (p_1(\ve{x}),\dots,p_N(\ve{x}))^T$ denotes the vector of basis monomials. In the following, we will also use the $K \times N$ matrices
\begin{equation}
\label{mat_P}
\ve{P_x} = \mat{c}{\ve{p}(\ve{x}_1)^T \\ \vdots \\ \ve{p}(\ve{x}_K)^T}
\qquad 
\ve{P_y} = \mat{c}{\ve{p}(\ve{y}_1)^T \\ \vdots \\ \ve{p}(\ve{y}_K)^T}\,.
\end{equation}

\subsubsection{Second step - identification of the Koopman operator}
\label{sec:ident_Koopman}

Now we proceed to the identification of the Koopman operator $U^t$, for $t=T_s$. More precisely, we will identify the finite-rank operator $U_N: \mathcal{F}_N \to \mathcal{F}_N$ of the form $U_N = P_N U^{T_s}|_{\mathcal{F}_N}$, where $P_N:\mathcal{F} \to \mathcal{F}_N$ is a projection operator onto the subspace $\mathcal{F}_N$ and where $U^t|_{\mathcal{F}_N}:\mathcal{F}_N \to \mathcal{F}$ is the restriction of the Koopman operator to $\mathcal{F}_N$. Considering
\begin{equation}
\label{f_in_basis}
f = \ve{a}^T\, \ve{p}\,, \qquad U_N f = \ve{b}^T \ve{p}\,, 
\end{equation}
we can define a matrix $\ve{\overline{U}}_N \in \mathbb{R}^{N \times N}$ such that
\begin{equation}
\label{def_U_N}
\ve{\overline{U}}_N \, \ve{a}  =\ve{b} \,.
\end{equation}
The matrix $\ve{\overline{U}}_N$ is a representation of the projected Koopman operator $U_N$. It also provides an approximate finite-dimensional linear description of the nonlinear system. This description is not obtained through local linearization techniques and is valid globally. 

It follows from \eqref{f_in_basis} and \eqref{def_U_N} that
\begin{equation}
\label{equality_a_U}
U_N f = U_N (\ve{a}^T  \ve{p})  = (\ve{\overline{U}}_N \ve{a})^T \ve{p}
\end{equation}
and, since \eqref{equality_a_U} holds for all $\ve{a}$, we have
\blue{
\begin{equation}
\label{equality_a_U_2}
[U_N p_1 \, \cdots \, U_N p_N]^T  = \ve{p}^T \ve{\overline{U}}_N\,,
\end{equation}
}
where the operator $U_N$ acts on each component of the vector $\ve{p}$. By considering each column separately, we obtain $P_N U^{T_s} p_j = U_N p_j = \ve{c}_j^T \ve{p}$, where $\ve{c}_j$ is the $j$th column  of $\ve{\overline{U}}_N$. This shows that each column of $\ve{\overline{U}}_N$ is related to the projection onto $\mathcal{F}_N$ of the image of a basis function $p_j$ through the Koopman operator $U^{T_s}$.

There are an infinity of possible projections $P_N$. We consider here a discrete orthogonal projection yielding the least squares fit at the points $\ve{x}_k$, $k=1,\dots,K$, with $K \geq N$:
\begin{equation}
\label{eq:proj}
P_N g = \underset{\tilde{g} \in \textrm{span}\{p_1,\dots,p_N\}}{\textrm{argmin}} \sum_{k=1}^K |\tilde{g}(\ve{x}_k) - g(\ve{x}_k)|^2\,.
\end{equation}
This corresponds to the least squares solution
\begin{equation*}
P_N g = \ve{p}^T \ve{P_x^\dagger} \mat{c}{g(\ve{x}_1) \\ \vdots \\ g(\ve{x}_K)}
\end{equation*}
where $\ve{P}^\dagger$ denotes the pseudoinverse of $\ve{P}$. For $g=U^{T_s} p_j$, we obtain
\begin{equation*}
P_N (U^{T_s} p_j) = \ve{p}^T \ve{P_x^\dagger} \mat{c}{U^{T_s} p_j(\ve{x}_1) \\ \vdots \\ U^{T_s} p_j(\ve{x}_K)} \approx \ve{p}^T \ve{P_x^\dagger} \mat{c}{p_j(\ve{y}_1) \\ \vdots \\ p_j(\ve{y}_K)}
\end{equation*}
where we used \eqref{f_y} evaluated at the states $\ve{x}_k$ and assumed that measurement noise $\| \epsilon \|$ is small. Equivalently, we have $U_N \ve{p}^T \approx \ve{p}^T \ve{P_x^\dagger} \, \ve{P_y}$ so that \eqref{equality_a_U_2} yields
\begin{equation}
\label{eq:least_sq_prob}
\ve{\overline{U}}_N \approx \ve{P_x^\dagger} \, \ve{P_y} \,.
\end{equation}

Inspired by \eqref{Ut_L}, we finally compute
\begin{equation}
\label{eq:L_data}
\ve{\overline{L}_{data}} = \frac{1}{T_s} \log (\ve{P_x^\dagger} \, \ve{P_y})\,,
\end{equation}
where the function $\log$ denotes the (principal) matrix logarithm. The matrix $\ve{\overline{L}_{data}}$ is an approximation of the matrix representation $\ve{\overline{L}}_N$ of $L_N = P_N L|_{\mathcal{F}_{N}}$, where $L_N f = \ve{p}^T (\ve{\overline{L}}_N \ve{a})$ for all $f = \ve{p}^T \ve{a}$. A rigorous justification is given in Section \ref{sec:proofs}.

\begin{remark}
Even with no measure noise, $\ve{\overline{L}_{data}}$ is only an approximation of $\ve{\overline{L}}_N$. Indeed, $\ve{\overline{L}_{data}}$ is the matrix representation of the finite-rank operator $\frac{1}{T_s} \log (P_N U^{T_s}|_{\mathcal{F}_N}) = \frac{1}{T_s} \log (P_N e^{L T_s}|_{\mathcal{F}_N}) \neq P_N L|_{\mathcal{F}_N}$. The two matrices $\ve{\overline{L}_{data}}$ and $\ve{\overline{L}}_N$ are identical only in the limit $N \rightarrow \infty$ and under some additional conditions related to the non-uniqueness of the matrix logarithm (see Section \ref{sec:proofs} \blue{for the details}). \hfill $\diamond$
\end{remark}

\blue{
\subsubsection{Third step - identification of the vector field}
\label{subsec:third_step}
We are now in position to identify the coefficients \mbox{$\ve{w}_k= (w_k^1 \cdots w_k^n)$} of the vector field. With the basis function $p_l(\mathbf{x}) = x_j$ where $x_j$ is the $j$th component of $\mathbf{x}$, we have 
\begin{equation*}
L_N p_l = P_N \left( \ve{F} \cdot \nabla p_l \right) = P_N F_j = F_j\,.
\end{equation*}
Since $L_N p_l = \ve{p}^T (\ve{\overline{L}}_N \ve{e}_l)$, it follows that
\begin{equation}
\label{eq:F_j}
F_j = \ve{p}^T (\ve{\overline{L}}_N \ve{e}_l) \approx \ve{p}^T (\ve{\overline{L}_{data}} \ve{e}_l)\,,
\end{equation}
i.e. the $l$th column of $\ve{\overline{L}_{data}}$ contains the estimates $\hat{w}_k^j$.
Equivalently, we have
\begin{equation}
\label{eq:coeff_main}
\hat{w}_k^j = \left[ \ve{\overline{L}_{data}} \right]_{kl} \,.
\end{equation}
}
 
\begin{remark}[Nonlinear least squares problem]
The identification problem could also be performed at the level of the Koopman semigroup. However solving the equality $\ve{\overline{U}} = e^{\ve{\overline{L}}T_s}$ (with a square matrix $\ve{\overline{L}}$) amounts to solving a (nonconvex) nonlinear least squares problem (as done in \cite{Riseth}). This might also be equivalent to solving the direct identification problem with an exact Taylor discretization of time-derivatives \cite{Kazantzis}.
\hfill $\diamond$
\end{remark}

\blue{
\begin{remark}
The vector field coefficients are obtained with $n$ columns of $\ve{\overline{L}_{data}}$ related to the monomials of degree $1$. Instead, we could use all columns $\ve{\overline{L}_{data}}$. In this case, the coefficients are the solutions to an overdetermined set of equations, which could be solved by promoting sparsity (e.g. Lasso). More details can be found in \cite{Mauroy_Goncalves_CDC}. However, numerical experiments suggest that this does not improve the results. \hfill $\diamond$
\end{remark}
}

\begin{remark}[Estimation of the vector field values]
\label{rem:estim_vec}
If needed, the method can directly provide the values $\ve{F}(\ve{x}_k)$ of the vector field. Evaluating \eqref{eq:F_j} at $\ve{x}_k$ for all $k=1,\dots,K$, we obtain an approximation $\hat{F}_j$ of the vector field given by
\begin{equation}
\label{eq:comput_vec_field}
\begin{pmatrix}
\hat{F}_j(\ve{x}_1) \\
\vdots\\ 
\hat{F}_j(\ve{x}_K)
\end{pmatrix}
= \ve{P_x} \, (\ve{\overline{L}_{data}} \, \ve{e}_l)
\end{equation}
with \blue{$p_l(\mathbf{x}) = x_j$.} This is quite similar to the approach developed with the dual method presented in Section \ref{sec:dual}. \hfill $\diamond$
\end{remark}

\subsection{Algorithm}

Our main lifting method for system identification is summarized in Algorithm \ref{alg:lifting}.

\begin{algorithm}[h]
	\caption{Main lifting method for nonlinear system identification}
	\label{alg:lifting}
	\begin{algorithmic}[1]
\Statex{\bf Input:} Snapshot pairs $\{(\ve{x}_k,\ve{y}_k)\}_{k=1}^K$, $\ve{x}_k \in \mathbb{R}^n$; sampling period $T_s$; \blue{integers $m \geq 1$ and $m_F \geq 0$ (with $m\geq m_F$).}
\Statex{\bf Output:} Estimates $\hat{w}_k^j$.
\State $N := (m+n)!/(m!n!)$; $N_F := (m_F+n)!/(m_F!n!)$
\While {$N > K$} 
					\State Increase $K$ (add snapshot pairs) or decrease $m$
\EndWhile
\State Construct the $K\times N$ matrices $\ve{P_x}$ and $\ve{P_y}$ defined in \eqref{mat_P}
\State Compute the $N \times N$ matrix $\ve{\overline{L}_{data}}$ defined in \eqref{eq:L_data}
\State \blue{$\hat{w}_k^j := \left[ \ve{\overline{L}_{data}} \right]_{kl}$, with $l$ such that $p_l(\mathbf{x}) = x_j$}
	\end{algorithmic}
\end{algorithm}

\subsection{Theoretical results}
\label{sec:proofs}

In this section, we prove the convergence of Algorithm \ref{alg:lifting} in optimal conditions, i.e. with an infinite number of data points and basis functions, and an arbitrarily high sampling rate.

We consider the space $\mathcal{F} = L^2(X)$  (where $\|\cdot \|$ is the $L^2$ norm) and the subspace $\mathcal{F}_N$ spanned by the monomials $\{p_k\}_{k=1}^N$. We will further assume that the flow induced by \eqref{eq:syst1} is invertible and nonsingular \footnote{The flow is nonsingular if $\mu(A) \neq 0$ implies $\mu(\varphi^t(A)) \neq 0$ for all $A \in \mathbb{R}^n$ and all $t>0$, where $\mu$ is the Lebesgue measure. This is a generic condition that is satisfied when the vector field $F$ is Lipschitz continuous, for instance.}, and that $X$ is forward-invariant (i.e. $\varphi^t(X) \subseteq X$ for all $t>0$) or backward-invariant\footnote{When $X$ is backward-invariant, we assume that $f(x)=0$ for all $x \notin X$ and all $f \in \mathcal{F}$, so that $U^t$ is a well-defined semigroup.} (i.e. $\varphi^{-t}(X) \subseteq X$ for all $t>0$). Under these conditions, we can check that the semigroup $U^t$ is strongly continuous. For continuous functions $g:X \to \mathbb{R}$, which are dense in $L^2(X)$, we have $\lim_{t \rightarrow 0} \| g - U^t g \| = 0$. Moreover, we have
\begin{equation*}
\begin{split}
\|U^t f \|^2 = \int_X |U^t f(x)|^2 dx & =  \int_{\varphi^t(X)} |f(x)|^2 |J_{\varphi^{-t}}(x)| dx \\
& \leq  \max_{x \in X} | J_{\varphi^{-t}}(x)| \|f \|^2
\end{split}
\end{equation*}
or equivalently
\begin{equation*}
\frac{\|U^t f \|^2}{\|f \|^2} \leq \max_{x \in X} |J_{\varphi^{t}}(x)|^{-1}
\end{equation*}
where $|J_{\varphi^{t}}(x)|$ is the determinant of the Jacobian matrix of $\varphi^t(x)$. Since the flow is nonsingular, $|J_{\varphi^{t}}(x)| \neq 0$ implies that $U^t$ is bounded. It follows that the semigroup of Koopman operators $U^t$ is strongly continuous (see e.g. \cite[Proposition I.5.3(c)]{Engel_Nagel}).

\blue{We are now in position to show that Algorithm \ref{alg:lifting} yields exact estimates $\hat{w}_k^j$ and $\hat{F}_j = \sum_{k=1}^{N} \hat{w}_k^j p_k$ in optimal conditions.}
\begin{theorem}
\label{theo1}
\blue{
Assume that the sample points $\ve{x}_k$ are uniformly randomly distributed in a compact forward or backward invariant set $X$, and consider $\ve{y}_k = \varphi^{T_s}(\ve{x}_k)$ (no measurement noise) where $\varphi^t$ is an invertible and nonsingular flow generated by the dynamics \eqref{eq:syst1}. If the Algorithm \ref{alg:lifting} is used with the data pairs $\{\ve{x}_k,\ve{y}_k\}_{k=1}^K$ (with $K\geq N$) and with a set of basis functions whose span  is dense in $L^2(X)$ and which contains the identity function $f_j(\ve{x})=x_j$, then the estimated vector field satisfies
\begin{equation*}
\lim_{N \rightarrow  \infty} \lim_{K \rightarrow  \infty} \lim_{T_s \rightarrow 0} \left\| \hat{F}_j - F_j \right\| = 0
\end{equation*}
with probability one. Moreover, if the vector field is of the form \eqref{eq:vec_field} with $h_k=p_k$ (monomials), then
\begin{equation*}
\lim_{T_s \rightarrow 0} \hat{w}_k^j = w_k^j\,, \qquad k=1,\dots,N_F
\end{equation*}
with probability one for all $N \geq N_F$.
}
\end{theorem}
\begin{proof}
\blue{
Since $\ve{x}_k \in X$, the discrete orthogonal projection \eqref{eq:proj} $P_N$ is a well-defined projection (with probability one) from $L^2(X)$ to $\mathcal{F}_N \subset L^2(X)$. For a finite integer $N$, consider the finite-dimensional operators $A_N^{T_s}=e^{P_N L P_N T_s}:\mathcal{F}_N\to \mathcal{F}_N$ and $U_N^{T_s}=P_N U^{T_s}|_{\mathcal{F}_N}$.
Since $\|A_N^{T_s} - I\| \rightarrow 0$ and $\|U_N^{T_s} - I\| \rightarrow 0$ as $T_s \rightarrow 0$, it follows that \begin{equation*}
\begin{split}
\lim_{T_s \rightarrow 0} \frac{\|\log A_N^{T_s} - (A_N^{T_s}-I) \|}{T_s} = 0 \\ \lim_{T_s \rightarrow 0} \frac{\|\log U_N^{T_s} - (U_N^{T_s}-I) \|}{T_s} = 0
\end{split}
\end{equation*}
(The eigenvalues $\lambda(T_s)$ of $A_N^{T_s}$ and $U_N^{T_s}$ satisfy \mbox{$|\lambda(T_s) - 1| \rightarrow 0$}, which implies by L'Hôpital's rule that \mbox{$(\log \lambda(T_s) - (\lambda(T_s)-1))/T_s \rightarrow 0$}.) 
We also have, for all $f\in \mathcal{F}_N$,
\begin{equation*}
\begin{split}
\left\|(A_N^{T_s} - U_N^{T_s} ) f \right\| & = \left \|\int_0^{T_s}  \frac{d}{d\tau} \left(A_N^\tau U_N^{T_s-\tau} \right) f  d\tau \right\|\\
& = \int_0^{T_s}  \left \| A_N^\tau (P_N L P_N-P_N L) U^{T_s-\tau} f \right\| d\tau \\
& \leq \int_0^{T_s}  \left \| A_N^\tau (P_N L P_N-P_N L) f \right\|  \\
& \quad +\left \| A_N^\tau P_N L P_N \right\| \left\| U^{T_s-\tau} f -f\right\|  \\
& \quad +\left \| A_N^\tau P_N \right\| \left\| U^{T_s-\tau} Lf -Lf\right\| d\tau
\end{split}
\end{equation*}
where we used the fact that $L$ and $U^{T_s-\tau}$ commute.
Since $U^{T_s}$ is strongly continuous and $P_n f = f$, it follows from the mean value theorem that
\begin{equation*}
\lim_{T_s \rightarrow 0} \frac{1}{T_s} \left\|(A_N^{T_s} - U_N^{T_s} ) f \right\| = 0 \qquad \forall f\in \mathcal{F}_N \,.
\end{equation*}
Then we get
\begin{equation*}
\begin{split}
& \lim_{T_s \rightarrow 0}\frac{1}{T_s}\|(\log A_N^{T_s} - \log U_N^{T_s} ) f \| \\
& \leq \lim_{T_s \rightarrow 0} \left( \frac{1}{T_s}\|(\log A_N^{T_s} - (A_N^{T_s}-I) ) f \| \right.\\
& \left. \quad + \frac{1}{T_s}\|(\log U_N^{T_s} - (U_N^{T_s}-I) ) f \| + \frac{1}{T_s}\|(A_N^{T_s} - U_N^{T_s} ) f \| \right) = 0 \,.
\end{split}
\end{equation*}
Since there is no measurement noise, \eqref{eq:L_data} and \eqref{eq:comput_vec_field}
imply that
\begin{equation}
\label{eq:diff_F_zero}
\begin{split}
\lim_{T_s \rightarrow 0} \left\| \hat{F}_j - P_N F_j \right\| & = \lim_{T_s \rightarrow 0} \left\| \frac{\log U^{T_s}_N f_j}{T_s}  - P_NLP_N f_j \right\| \\
& = \lim_{T_s \rightarrow 0} \frac{1}{T_s} \left\| \log U^{T_s}_N f_j - \log A_N^{T_s} f_j \right\| = 0 \\
\end{split}
\end{equation}
with the identity function $f_j\in \mathcal{F}_N$.
}

\blue{The discrete orthogonal projection converges in the strong operator topology with probability one to the $L^2$ projection (see e.g. \cite{Korda_convergence} for a proof). Since the basis is complete in $L^2(X)$, the orthogonal projection converges in the strong operator topology to the identity operator as $N \rightarrow \infty$, with probability one. It follows that
\begin{equation*}
\begin{split}
& \lim_{N \rightarrow  \infty} \lim_{K \rightarrow  \infty} \lim_{T_s \rightarrow 0} \left\| \hat{F}_j - F_j \right\| \\
& \leq \lim_{N \rightarrow  \infty} \lim_{K \rightarrow  \infty} \left(\lim_{T_s \rightarrow 0} \left\| \hat{F}_j - P_N F_j \right\| + \left\| P_N F_j - F_j \right\| \right) = 0 \, .
\end{split}
\end{equation*}
}

\blue{
Finally, if the vector field is polynomial with \mbox{$N_F \leq N$}, we have $P_N F_j=F_j$ (with probability one) and it follows from \eqref{eq:diff_F_zero} that $\lim_{T_s \rightarrow 0} \hat{w}_k^j - w_k^j=0$ for all $k$.
}
\end{proof}

According to Theorem \ref{theo1}, Algorithm \ref{alg:lifting} identifies exactly the vector field, even if the data are collected in a small region of the state space (this is made possible by the a priori knowledge that the vector field is polynomial). Note that the requirement to collect the data points on an invariant set might not always be satisfied in practice. This is however a technical condition that ensures that $U^t$ is a well-defined semigroup of operators on $[0,T_s]$. However, the result shows that an infinite sampling frequency ($T_s \rightarrow 0$) is required by the use of the matrix logarithm. This issue is related to the so-called system aliasing and is discussed with more details in \cite{Zuogong_CDC}. Intuitively, an infinite sampling rate is needed to capture the infinity of frequencies that characterize a nonlinear system. \blue{This condition ensures in particular that the eigenvalues of $T_s P_N L P_N$ lie in the strip $\{z\in \mathbb{C}:|\Im\{z\}|<\pi\}$ so that the properties of the principal branch of the logarithm imply that $\frac{1}{T_s} \log e^{T_s P_N L P_N} = P_N L P_N$ in \eqref{eq:diff_F_zero}. In the case of polynomial vector fields, it is noticeable that the number of basis functions does not need to tend to infinity. In fact, when $T_s$ tends to zero, $\log U_N^{T_s} f_j/T_s \approx (I-U_N^{T_s})f_j/T_s$ corresponds to the first order approximation of the time derivative $\dot{x}_j=F_j$ and we recover a direct method based on the computation of time derivatives. In practice, with a possibly large sampling time, it can be useful to increase the number of basis functions $N$. The Trotter-Kato approximation theorem (see e.g. \cite[Theorem 4.8]{Engel_Nagel}) implies that $\| (A_N^{T_s}-U_N^{T_s}) f \| \rightarrow 0$ as $N \rightarrow \infty$ for all $f \in \mathcal{F}_N$, so that one can expect that the error $\| (\log A_N^{T_s}- \log U_N^{T_s}) f_j \|$ decreases for larger values $N$. This is confirmed with numerical simulations suggesting that small estimation errors can be obtained with a sampling period $T_s = \mathcal{O}(1)$ provided that $N$ is large enough.}



The above theoretical results are valid only when there is no measurement noise. In presence of noise, the estimator is biased and not consistent, because of the lifting of the data.
However, the algorithm performs well for small measurement noise levels and is also shown to be robust to process noise in Section \ref{sec:examples}.

\section{Extensions}
\label{sec:extensions}

We now consider several extensions of the proposed method, which allow to identify open systems driven by a known input or a white noise (i.e. process noise) and to identify systems with non-polynomial vector fields.

\subsection{Systems with inputs}

Consider an open dynamical system of the form
\begin{equation}
\label{input_syst}
\dot{\ve{x}} = \ve{F}(\ve{x},\ve{u}(t))
\end{equation}
with $\ve{x} \in \mathbb{R}^n$ and with the input $u \in \mathcal{U}:\mathbb{R}^+ \to \mathbb{R}^p$. We assume that the vector field consists of monomials in $\ve{x}$ and $\ve{u}$. We define the associated flow $\varphi:\mathbb{R}^+ \times \mathbb{R}^n \times \mathcal{U}$ so that $t \mapsto \varphi(t,\ve{x},\ve{u(\cdot)})$ is a solution of \eqref{input_syst} with the initial condition $\ve{x}$ and the input $\ve{u(\cdot)}$. Following the generalization proposed in \cite{Brunton_ident_input,Proctor_input,Proctor_DMD}, we consider observables $f:\mathbb{R}^n \times \mathbb{R}^p \to \mathbb{R}$ and define the semigroup of Koopman operators
\begin{equation*}
U^t f(\ve{x,u}) = f(\varphi^t(\ve{x,u(\cdot) = \ve{u}}),\ve{u})
\end{equation*}
where $\ve{u(\cdot)} = \ve{u}$ is a constant input. In this case, $\ve{u}$ can be considered as additional state variables and the above operator is the classic Koopman operator for the augmented system $\dot{\ve{x}}=\ve{F(x,u)}$, $\dot{\ve{u}} = \ve{0}$. In particular, the infinitesimal generator is still given by \eqref{inf_gen}.

It follows that the method proposed in Sections \ref{sec:lifting_data} and \ref{sec:ident_Koopman} can be used if 
\begin{equation*}
\varphi^{T_s} (\ve{x,u(\cdot)}) \approx \varphi^{T_s} (\ve{x,u(\cdot)=u}(0))\,.
\end{equation*}
This condition holds when the input can be considered as constant between two snapshots (zero-order hold assumption), or equivalently if the sampling rate is high enough. The matrix $\ve{\overline{U}}_N$ is now obtained with snapshot pairs $([\ve{x}_k,\ve{u}_k],[\ve{y}_k,\ve{u}_k]) \in \mathbb{R}^{(n+p) \times 2}$ and the rest of the procedure follows on similar lines with the augmented state space $\mathbb{R}^{n+p}$. In this case, the identification method not only provides the vector field coefficients associated with the state $\ve{x}$, but also those associated with the input $\ve{u}$. The efficiency of the method is illustrated in Section \ref{sec:example_extension}.

\subsection{Process noise}

We have considered so far only measurement noise. We show that the proposed method is also robust to process noise. Consider a system described by the stochastic differential equation
\begin{equation}
\label{syst_stoch}
dx = F(\ve{x}) dt + \sigma \, dw(t)
\end{equation}
where $w(t)$ is the Wiener process. We define the flow $\varphi:\mathbb{R}^+ \times \mathbb{R}^n \times \Omega \to \mathbb{R}^n$, where $\Omega$ is the probability space, such that $t \mapsto \varphi(t,\ve{x},\omega)$ is a solution to \eqref{syst_stoch}. In this case, the semigroup of Koopman operators is defined by (see e.g. \cite{Mezic})
\begin{equation*}
U^t f (\ve{x}) = \mathbb{E}[f(\varphi(t,\ve{x},\omega))]
\end{equation*}
and its infinitesimal generator is given by
\begin{equation}
\label{eq:gen_stoch}
L f = \ve{F} \cdot \nabla f + \frac{\sigma^2_{proc}}{2} \Delta f
\end{equation}
where $\Delta=\sum_k \partial^2/\partial x_k^2$ denotes the Laplacian operator that accounts for diffusion. The infinitesimal generator is related to the so-called Kolmogorov backward equation.

The numerical scheme of the proposed identification method does not need to be adapted to take process noise into account. As explained in \cite{Rowley_EDMD}, the first step of the method (Section \ref{sec:lifting_data}) is still valid for identifying the matrix $\ve{\overline{U}}$. In the second step (Section \ref{sec:ident_Koopman}), the procedure is the same, except that one has to consider the infinitesimal generator whose matrix representation is given by $\ve{\overline{L}}  + \sigma^2/2 \ve{\overline{D}}$, where $\ve{\overline{D}}$ is the matrix representation of the Laplacian operator. For all $l$ such that $p_l(\ve{x})=x_j$ for some $j$, $\Delta p_l=0$ so that the $l$th column of $\ve{\overline{D}}$ contains only zeros. It follows that we can still use \eqref{eq:coeff_main} to compute the vector field coefficients. In Section \ref{sec:example_extension}, an example illustrates the robustness of the method against process noise.

Similar methods are also considered with manifold learning techniques (e.g. diffusion maps) for state estimation \cite{Shnitzer} and embedded vector field estimation \cite{Perrault}. Although developed in another context to solve different problems, these techniques are similar to our method in the sense that they rely on the backward Kolmogorov equation, which is directly connected to the infinitesimal generator \eqref{eq:gen_stoch}.

\subsection{Non polynomial vector fields}
\label{sec:other_vec}

The method can be adapted to identify non polynomial vector fields of the form \eqref{eq:vec_field}, where the library functions $h_k$ are not monomials. In this context, the vector field and the library functions do not need to be analytic. In this case, one could consider the equality
\begin{equation*}
P_N L = \sum_{j=1}^n \sum_{k=1}^{N_F} w^j_k \, P_N L^j_k
\end{equation*}
with the operators $L^j_k = h_k \partial/\partial x_j$. The final-dimensional representation of this equality yields a matrix equation that should be solved to compute the coefficients $w^j_k$.

However, using the projection $P_N$ adds an additional error to the finite-dimensional approximation of the operator. Instead, we prefer to consider an \guillemets{augmented} subspace that contains the library functions:
\begin{equation*}
\mathcal{F}_{N}' = \mathcal{F}_{N} 
\times \textrm{span}\left(\left\{h_k\right\}_{k=1}^{N_F} \right)
\end{equation*}
where $\mathcal{F}$ is a subspace spanned by monomials. In this case, we can still use Algorithm \ref{alg:lifting} with the projection \mbox{$P'_N:\mathcal{F} \to \mathcal{F}_{N}'$.} \blue{The result of Theorem \ref{theo1} could also be extended to this case, provided that the set of basis functions is complete in $L^2$.}

We finally note that a more straightforward method is to perform a least squares regression on the values of the vector field at the sample points, values which can be obtained according to Remark \ref{rem:estim_vec}. However, numerical experiments suggest that this method is less efficient than the above-mentioned method.

\section{A dual lifting method for large systems}
\label{sec:dual}

A major limitation of the main method presented in Section \ref{sec:lifting} (Algorithm \ref{alg:lifting}) is that it might require a large number of data points. Indeed, the number of data points must be larger than the number of basis functions ($K \geq N$) to ensure that the discrete orthogonal projection \eqref{eq:proj} is well-defined. In the case of high-dimensional systems in particular, the number of basis functions is huge and is likely to exceed the number of available data points, an issue which might be critical in fields such as biology. Moreover, the algorithm might also be computationally intractable (e.g. computation of the matrix logarithm in \eqref{eq:L_data}). In this section, we circumvent the above limitations by proposing a dual approach, which is developed in a $K$-dimensional ``sample space'' instead of the $N$-dimensional functional space. This method can be used when the number of basis functions is larger than the number of data points, i.e. $N \geq K$.

\subsection{Description of the method}

Similarly to the main lifting method, the dual method consists of three steps: lifting of the data, identification of the Koopman operator, and identification of the vector field. In the last step, the algorithm provides the value of the vector field at each data point, so that the dual method can be seen as an indirect method for time derivatives estimation. This is similar in essence to the vector field estimation detailed in Remark \ref{rem:estim_vec}. The identification is achieved in a distributed way, a feature which makes the algorithm computationally tractable in the case of high-dimensional systems and well-suited to parallel computing.

\subsubsection{First step - lifting of the data} This step is similar to the first step of the main method (Section \ref{sec:lifting_data}). But in this case, choosing the basis functions equal to the library functions of the vector field is not more convenient for the next steps. Even if the vector field is polynomial, we can therefore consider other basis than monomials, such as Gaussian radial basis functions $g_k(\ve{x}) = e^{-\gamma \|\ve{x}-\ve{x}_k\|^2}$ with $k=1,\dots,K$ and where $\gamma>0$ is a parameter. We construct the data $K \times N$ matrices
\begin{equation}
\label{mat_P_dual}
\ve{P_x} = \mat{c}{\ve{g}(\ve{x}_1)^T \\ \vdots \\ \ve{g}(\ve{x}_K)^T}
\qquad 
\ve{P_y} = \mat{c}{\ve{g}(\ve{y}_1)^T \\ \vdots \\ \ve{g}(\ve{y}_K)^T}
\end{equation}
where $\ve{g}$ is the vector of basis functions $g_k$. When using Gaussian radial basis functions, the number of basis functions is equal to the number of samples (i.e. $N=K$) and therefore does not depend on the dimension $n$. This is particularly useful in the case of high-dimensional systems, where the matrices \eqref{mat_P_dual} should be of reasonable size. In Section \ref{sec:examples}, we will only use Gaussian radial basis functions.

\subsubsection{Second step - identification of the Koopman operator} We use a dual matrix representation of the Koopman operator, which is inspired (but slighted different, see Remark \ref{rem:kernel}) from a kernel-based approach developed in \cite{Williams_kernel} (kernel EDMD).

In the main method, we constructed the $N\times N$ matrix $\ve{\overline{U}}_N \approx \ve{P_x^\dagger} \, \ve{P_y}$ which represents the operator $U_N$. Instead, we can consider the $K \times K$ matrix representation
\begin{equation}
\label{eq:hat_U}
\ve{\widetilde{U}}_K \approx \ve{P_y} \, \ve{P_x^\dagger} = \ve{P_x} \, \ve{\overline{U}}_N \, \ve{P_x^\dagger} \,,
\end{equation}
a construction which is similar to the original formulation of the Dynamic Mode Decomposition (DMD) algorithm\footnote{This would correspond exactly to DMD if the basis functions $g_j$ were replaced by functions $g_j(x)=\varphi^{(j-1) t_s}(x)$.} \cite{Tu}. The matrix $\ve{P_x}$ can be interpreted as a change of coordinates, and $\ve{\widetilde{U}}_K$ appears to be the matrix representation of $U^{T_s}$ in the ``sample space'': for all $f \in \mathcal{F}_N$, we have
\begin{equation}
\label{eq:Utilde_def}
\mat{c}{U^{T_s} f(\ve{x}_1) \\ \vdots \\ U^{T_s} f(\ve{x}_K)} \approx \ve{\widetilde{U}}_K \mat{c}{f(\ve{x}_1) \\ \vdots \\ f(\ve{x}_K)} \,.
\end{equation}

We have seen that the $j$th column $\ve{c}_j$ of $\ve{\overline{U}}_N$ satisfies $\ve{P_x} \ve{c}_j \approx (p_j(\ve{y}_1) \, \cdots \, p_j(\ve{y}_K))^T$ and corresponds to the projection \eqref{eq:proj} of $U^{T_s} p_j$ on $\mathcal{F}_N$ (expressed in the basis of functions). Each of the $K$ data points yields a constraint and there are $N$ unknowns, so that $K \geq N$ is required. In contrast, it follows from \eqref{eq:Utilde_def} that the $i$th row $\ve{r}_i$ of $\ve{\widetilde{U}}_K$ can be seen, for all $f$, as the coefficients of the linear combination of the values $f(\ve{x}_1),\dots,f(\ve{x}_K)$ that is equal to $U^{T_s} f(\ve{x}_i)$. 
The row $\ve{r}_i$ satisfies $r_i \ve{P_x} \approx (g_1(\ve{y}_i) \, \cdots \, g_N(\ve{y}_i))$, i.e. $r_i$ is obtained by considering the $N$ \guillemets{test} functions $g_j$. In this case, each of the $N$ functions yields a constraint and there are $K$ unknowns, so that $K \leq N$ is required.


\begin{remark}
\label{rem:kernel}
Following similar lines as in \cite{Williams_kernel}, we note that we have
\begin{equation*}
\ve{\widetilde{U}}_K \approx \ve{P_y} \ve{P_x^\dagger} = \ve{P_y} \ve{P_x^T} (\ve{P_x}\ve{P_x^T})^\dagger \triangleq \ve{A \, G^\dagger}
\end{equation*}
where the entries of $\ve{A}$ and $\ve{G}$ can be interpreted as the inner products
\begin{equation*}
[\ve{A}]_{ij} = \ve{p}(\ve{x}_j)^T \ve{p}(\ve{y}_i)\, , \quad [\ve{G}]_{ij} = \ve{p}(\ve{x}_j)^T \ve{p}(\ve{x}_i)
\end{equation*}
(Here, we consider without loss of generality that the matrices $\ve{P_x}$ and $\ve{P_y}$ are constructed with monomials.) The inner products can be approximated by a Gaussian kernel function $g(\ve{x}_i,\ve{x}_j)=g_j(\ve{x}_i)$, so that 
\begin{equation*}
[\ve{A}]_{ij} = g(\ve{x}_i,\ve{x}_j) \, , \quad [\ve{G}]_{ij} = g(\ve{y}_i,\ve{x}_j) \,.
\end{equation*}
In this context, constructing $\ve{P_x}$ and $\ve{P_y}$ with Gaussian radial basis functions is equivalent to constructing the inner-product matrices $\ve{A}$ and $\ve{G}$.\\
At this point, we can note that our matrix representation $\ve{\widetilde{U}}_K$ is slightly different from the representation used for kernel EDMD in \cite{Williams_kernel}, which is given by
\begin{equation*}
\ve{G^\dagger \, A} = (\ve{P_x^T})^\dagger \ve{P}_\ve{x}^\dagger \ve{P_y} \ve{P_x^T} \neq \ve{\widetilde{U}}_K\,.
\end{equation*}
\hfill $\diamond$
\end{remark}

Finally, similarly to \eqref{eq:L_data}, we compute the $K \times K$ matrix
\begin{equation}
\label{eq:L_data_hat}
\ve{\widetilde{L}_{data}} = \frac{1}{T_s} \log (\ve{P_y} \, \ve{P_x^\dagger})\,.
\end{equation}

\subsubsection{Third step - identification of the vector field}

Using a similar idea as the one explained in Remark \ref{rem:estim_vec}, we can directly identify the vector field at the different values $\ve{x}_k$ and the coefficients $w_j^k$ are then obtained by solving $n$ separate regression problems.

\paragraph{Computation of the vector field $\ve{F}(\ve{x}_k)$} We assume that $\ve{\widetilde{L}_{data}}$ is an approximation of the matrix representation of $L$ in the sample space and we have
\begin{equation*}
\mat{c}{ \hspace{-0.2cm} \ve{F}(\ve{x}_1) \cdot \nabla f(\ve{x}_1) \\ \vdots  \\ \hspace{-0.2cm} \ve{F}(\ve{x}_K) \cdot \nabla f(\ve{x}_K)} = \mat{c}{ \hspace{-0.2cm} L f(\ve{x}_1) \hspace{-0.2cm} \\ \vdots \\ \hspace{-0.2cm}  L f(\ve{x}_K) \hspace{-0.2cm} } \approx \ve{\widetilde{L}_{data}} \mat{c}{\hspace{-0.2cm} f(\ve{x}_1) \\ \vdots \\ \hspace{-0.2cm} f(\ve{x}_K)} \,.
\end{equation*}
Considering the above equality with the identity function $\ve{f}(\ve{x})=\ve{x}$, we obtain an approximation $\ve{\hat{F}}$ of the vector field that is given by
\begin{equation}
\label{eq:ident_vec_dual}
\mat{c}{\ve{\hat{F}}(\ve{x}_1)^T \\ \vdots \\ \ve{\hat{F}}(\ve{x}_K)^T } \approx \ve{\widetilde{L}_{data}} \mat{c}{\ve{x}_1^T \\ \vdots \\ \ve{x}_K^T } \,.
\end{equation}
The choice of the functions $f$ used to obtain \eqref{eq:ident_vec_dual} is arbitrary. However, considering monomials of degree one is natural and choosing more functions would yield an overconstrained problem which does not necessarily improve the accuracy of the result. Note also that an approach more similar to the main method is to compute (an approximation of) the matrix representation of $L=\ve{F} \cdot \nabla$ in the sampling space and compare it with $\ve{\widetilde{L}_{data}}$. However, this does not yield better results.

\paragraph{Computation of the coefficients $w_k^j$} When the value of the vector field is known at every data points, we can find an estimation $\hat{w}_k^j$ of the coefficients $w_k^j$ by solving a regression problem. This problem is decoupled: for each $j=1,\dots,n$, we have to solve
\begin{equation*}
\hat{F}_j(\ve{x}_k) = \sum_{l=1}^{N_F} \hat{w}_j^l \, h_l(\ve{x}_k) \qquad k=1,\dots,K\,,
\end{equation*}
which takes the form
\begin{equation}
\label{eq:regression}
\mat{c}{\hat{F}_j(\ve{x}_1) \\ \vdots \\ \hat{F}_j(\ve{x}_{K})} = \ve{H_x} \mat{c}{\hat{w}_1^j \\ \vdots \\ \hat{w}_{N_F}^j}
\end{equation}
with
\begin{equation}
\label{mat_H_x}
\ve{H_x} = \mat{c}{\ve{h}(\ve{x}_1)^T \\ \vdots \\ \ve{h}(\ve{x}_K)^T}
\end{equation}
and where $\ve{h}$ is the vector of library functions $h_k$ of the vector field. Since we do not make any assumption on the vector field, which might not be polynomial or even analytic, the library functions are not necessarily monomials.

Since we can reasonably assume that most coefficients are zero, we can promote sparsity of the vector of coefficients $\hat{w}_j^l$ by adding a penalty term, which yields the Lasso optimization problem \cite{Lasso}
\begin{equation}
\label{eq:Lasso}
\min_{\ve{w} \in \mathbb{R}^{N_F}} \left \|\ve{H_x} \ve{w} - \mat{c}{\hat{F}_j(\ve{x}_1) \\ \vdots \\ \hat{F}_j(\ve{x}_{K})} \right\|_2^2 + \rho \|\ve{w}\|_1 
\end{equation}
where $\rho$ is a positive regularization parameter. Other techniques could also be used to infer $\ve{w}$ from the values of the vector field (see e.g. \cite{Brunton, Wei}). More generally, machine learning techniques could also be used to solve the regression problem \eqref{eq:regression}.

\subsection{Algorithm}

The dual method is summarized in Algorithm \ref{alg:dual}.
\begin{algorithm}[h]
	\caption{Dual lifting method for nonlinear system identification}
	\label{alg:dual}
	\begin{algorithmic}[1]
\Statex{\bf Input:} Snapshot pairs $\{(\ve{x}_k,\ve{y}_k)\}_{k=1}^K$, $\ve{x}_k \in \mathbb{R}^n$; basis functions $\{g_k\}_{k=1}^N$ (with $N \geq K$) ; library functions $\{h_k\}_{k=1}^{N_F}$.
\Statex{\bf Output:} Estimates $\ve{\hat{F}}(\ve{x}_k)$ and $\hat{w}_k^j$.
\State Construct the $K\times N$ matrices $\ve{P_x}$ and $\ve{P_y}$ defined in \eqref{mat_P_dual}
\State Compute the $K \times K$ matrix $\ve{\widetilde{L}_{data}}$ defined in \eqref{eq:L_data_hat}
\State Obtain $\ve{\hat{F}}(\ve{x}_k)$ with \eqref{eq:ident_vec_dual}
\State Construct the $K \times N_F$ matrices $\ve{H_x}$ defined in \eqref{mat_H_x}
\State For each $j$, solve the regression problem \eqref{eq:regression}, e.g. solve the Lasso problem \eqref{eq:Lasso}, to obtain $\hat{w}_k^j$
	\end{algorithmic}
\end{algorithm}

\subsection{Theoretical results}
\label{sec:theo_dual}

We now show the convergence of Algorithm \ref{alg:dual} in optimal conditions. \blue{Let $X \subset \mathbb{R}^n$ be a compact set and $\mathcal{F} = C(X)$ (where $\| \cdot \|$ is the $L^\infty$ norm). It is easy to verify that the Koopman semigroup $U^t: \mathcal{F} \to \mathcal{F}$ is strongly continuous. Assuming that the data points $\ve{x}_k \in X$ are uniformly randomly distributed, we consider the linear functionals $\xi_k: f \mapsto \xi_k(f)=f(\ve{x}_k)$ which span a subspace $\widetilde{\mathcal{F}}_K^*$ of the dual space $\mathcal{F}^*$ of $\mathcal{F}$. We can define the discrete projection operator $\widetilde{P}_K^*:\mathcal{F}^* \to \mathcal{F}^*_K$ by
\begin{equation}
\label{eq:proj_dual}
\widetilde{P}_K^* \xi = \underset{\tilde{\xi} \in \textrm{span}\{\xi_1,\dots,\xi_K\}}{\textrm{argmin}} \sum_{l=1}^N |\xi(g_l) - \tilde{\xi}(g_l)|^2\,.
\end{equation}
The adjoint (projection) operator $\widetilde{P}_K:\mathcal{F} \to \mathcal{F}$ of $\widetilde{P}_K^*$ (i.e. $\widetilde{P}_K^* \xi(f) = \xi(\widetilde{P}_K f)$ for all $f\in \mathcal{F}$, $\xi \in \mathcal{F}^*$) satisfies
\begin{equation*}
(\widetilde{P}_K f)(\ve{x}_k) = \xi_k(\widetilde{P}_K f) = \widetilde{P}^*_K \xi_k(f) = \xi_k(f)= f(\ve{x}_k)\,.
\end{equation*}
Moreover, for all $\xi \neq 0$ such that $\xi(g_l)=0$ $\forall l$, we have $\widetilde{P}^*_K \xi = 0$ and equivalently $\xi(\widetilde{P}_K f)=0$ for all $f\in \mathcal{F}$ so that $\widetilde{P}_K f \in \textrm{span}\{g_l\}_{l=1}^N$. It follows that \mbox{$\widetilde{P}_K:\mathcal{F} \to \widetilde{\mathcal{F}}_K$} is a projection onto a subspace $\widetilde{\mathcal{F}}_K \subseteq \textrm{span}\{g_l\}_{l=1}^N$ that is obtained through interpolation with collocation points $\ve{x}_k$. Finally the matrix $\widetilde{\ve{U}}_K$ can be interpreted as the matrix representation of the finite-rank operator \mbox{$\widetilde{U}^*_K=\widetilde{P}_K^* (U^{T_s})^* |_{\widetilde{\mathcal{F}}_K^*}:\widetilde{\mathcal{F}}_K^* \to \widetilde{\mathcal{F}}_K^*$} in the sense that $\widetilde{U}^*_K \xi = \ve{c}^T \widetilde{\ve{U}}_K \Xi$ for all $\xi = \ve{c}^T \Xi \in \widetilde{\mathcal{F}}_K^*$ (with $\Xi=(\xi_1 \cdots \xi_K)^T$). It follows that we have
\begin{equation*}
\mat{c}{U^{T_s} f(\ve{x}_1) \\ \vdots \\ U^{T_s}f(\ve{x}_K)} \approx \mat{c}{\widetilde{U}^*_K \xi_1(f) \\ \vdots \\ \widetilde{U}^*_K \xi_K(f)} = \widetilde{\ve{U}}_K \mat{c}{f(\ve{x}_1) \\ \vdots \\ f(\ve{x}_K)}
\end{equation*}
and we recover \eqref{eq:Utilde_def}. Equivalently, the matrix $\widetilde{\ve{U}}_K$ can be seen as the representation of $\widetilde{P}_K U^{T_s} |_{\widetilde{\mathcal{F}}_K}:\widetilde{\mathcal{F}}_K \to \widetilde{\mathcal{F}}_K$ with the basis functions $f_k \in \widetilde{\mathcal{F}}_K$ such that $f_k(\ve{x}_j)=\delta_{kj}$.}

\blue{
We will assume that
\begin{equation}
\label{eq:cond_gl}
\mu \left\{\ve{x} \in X \,|\,\sum_{l=1}^\infty c_l g_l(\ve{x}) = 0 \right\} = 0 \quad \forall (c_1,c_2,\dots) \neq \ve{0}\,,
\end{equation}
where $\mu$ denotes the Lebesgue measure. This independence condition (see also \cite{Korda_convergence}) ensures that the projections $\widetilde{P}_K^*$ and $\widetilde{P}_K$ are well-defined with probability one as $K\rightarrow \infty$.
}

The following result shows that Algorithm \ref{alg:dual} provides exact estimates $\hat{\ve{F}}(\ve{x}_k)$ of the vector field in optimal conditions.
\begin{theorem}
\label{theo2}
Assume that the sample points $\ve{x}_k$ are uniformly randomly distributed in a compact forward invariant set $X$, and consider $\ve{y}_k = \varphi^{T_s}(\ve{x}_k)$ (no measurement noise) where $\varphi^t$ is a (invertible and nonsingular) flow generated by the dynamics $\ve{\dot{x}}=\ve{F(x)}$.

\blue{If Algorithm \ref{alg:dual} is used with $N=K$ basis functions \mbox{$g_l \in C^1(X)$} such that \eqref{eq:cond_gl} holds and the identity function \mbox{$f_j(\ve{x})=x_j$} is in the span of $\{g_l\}_{l=1}^N$, then
\begin{equation*}
\lim_{T_s \rightarrow 0} \left|F_j(\ve{x}_k)-\hat{F}_j(\ve{x}_k)\right| = 0 \quad \forall k
\end{equation*}
with probability one.
}

\blue{
If Algorithm \ref{alg:dual} is used with $N \geq K$ basis functions \mbox{$g_l \in C^1(X)$} such that \eqref{eq:cond_gl} holds and
$\lim_{K \rightarrow \infty} \lim_{N \rightarrow \infty} \|\widetilde{P}_K f_j-f_j\|_L=0$
(with the graph norm $\|f\|_L=\|f\|+\|Lf\|$), then
\begin{equation*}
\lim_{K \rightarrow \infty} \lim_{N \rightarrow \infty} \lim_{T_s \rightarrow 0} \left|F_j(\ve{x}_k)-\hat{F}_j(\ve{x}_k)\right| = 0 \quad \forall k
\end{equation*}
with probability one.
}
\end{theorem}
\begin{proof}
\blue{It follows from \eqref{eq:cond_gl} that the discrete projection $\widetilde{P}_K$ is well-defined (with probability one). For a finite integer $K$, consider the finite-dimensional operators $\widetilde{A}_K^{T_s}=e^{\widetilde{P}_K L \widetilde{P}_K T_s}:\widetilde{\mathcal{F}}_K \to \widetilde{\mathcal{F}}_K$ and $\widetilde{U}_K^{T_s}=\widetilde{P}_K U^{T_s}|_{\widetilde{\mathcal{F}}_K}$. We have
\begin{equation*}
\begin{split}
\lim_{T_s \rightarrow 0} \frac{\|\log \widetilde{A}_K^{T_s} - (\widetilde{A}_K^{T_s}-I) \|}{T_s} = 0 \\
\lim_{T_s \rightarrow 0} \frac{\|\log \widetilde{U}_K^{T_s} - (\widetilde{U}_K^{T_s}-I) \|}{T_s} = 0
\end{split}
\end{equation*}
and, for all $f\in \widetilde{\mathcal{F}}_K$,
\begin{equation*}
\begin{split}
& \lim_{T_s \rightarrow 0} \frac{1}{T_s} \left\|(\widetilde{A}_K^{T_s} - \widetilde{U}_K^{T_s} ) f \right\| \\
& \quad = \lim_{T_s \rightarrow 0} \frac{1}{T_s} \int_0^{T_s} \left \| \frac{d}{d\tau} \left(\widetilde{A}_K^{\tau} \widetilde{U}_K^{T_s-\tau} \right) f \right\| d\tau \\
& \quad = \lim_{T_s \rightarrow 0} \frac{1}{T_s} \int_0^{T_s}  \left \| \widetilde{A}_K^{\tau} (\widetilde{P}_K L \widetilde{P}_K-\widetilde{P}_K L) U^{T_s-\tau}  f \right\| d\tau  = 0
\end{split}
\end{equation*}
for all $f\in \mathcal{F}$ since $U^{T_s}$ is strongly continuous (see the details in the proof of Theorem \ref{theo1}). This implies that
\begin{equation*}
\lim_{T_s \rightarrow 0}\frac{1}{T_s}\|(\log \widetilde{A}_K^{T_s} - \log \widetilde{U}_K^{T_s} ) f \| = 0 \qquad \forall f\in \widetilde{\mathcal{F}}_K
\end{equation*}
and it follows that 
\begin{equation}
\label{eq:F_xk_theo2}
\begin{split}
& \lim_{T_s \rightarrow 0} \left| \hat{F}_j(\ve{x}_k) - L \widetilde{P}_K f_j(\ve{x}_k) \right| \\
& \quad \leq \lim_{T_s \rightarrow 0} \left\| \frac{\log \widetilde{U}_K^{T_s} \widetilde{P}_K f_j}{T_s} - \widetilde{P}_K L \widetilde{P}_K f_j \right\| \\
& \quad = \lim_{T_s \rightarrow 0} \frac{1}{T_s} \left\| \log \widetilde{U}_K^{T_s} \widetilde{P}_K f_j - \log \widetilde{A}_K^{T_s} \widetilde{P}_K f_j \right\| = 0
\end{split}
\end{equation}
where $f_j$ is the identity function (with $f_j(\ve{x})=x_j$) and where we used $\widetilde{P}_K f(\ve{x}_k) = f(\ve{x}_k)$. If $K=N$ and since the vectors $(\xi_1(g_l) \cdots \xi_K(g_l))$ are linearly independent with probability one (this follows from \eqref{eq:cond_gl}, see also \cite{Korda_convergence}), \eqref{eq:proj_dual} implies that $\widetilde{P}_K^*\xi(g_l) = \xi(g_l)$ for all $l$, or equivalently $\xi(P_K g_l) = \xi(g_l)$ for all $\xi \in \mathcal{F}^*$, so that $P_K g_l = g_l$  by the Hahn-Banach theorem. If $f_j \in \textrm{span} \{g_l\}_{l=1}^N$, then $\widetilde{P}_K f_j = f_j$ and it follows from \eqref{eq:F_xk_theo2} that
\begin{equation*}
\lim_{T_s \rightarrow 0} \left| \hat{F}_j(\ve{x}_k) - F_j(\ve{x}_k) \right| = \lim_{T_s \rightarrow 0} \left| \hat{F}_j(\ve{x}_k) - L \widetilde{P}_K f_j(\ve{x}_k) \right| = 0 \,.
\end{equation*}
}

\blue{If $\lim_{K \rightarrow \infty} \lim_{N \rightarrow \infty} \|\widetilde{P}_K f_j - f_j \|_L=0$, we have
\begin{equation*}
\begin{split}
& \lim_{K \rightarrow \infty} \lim_{N \rightarrow \infty} \lim_{T_s \rightarrow 0} \left| \hat{F}_j(\ve{x}_k) - F_j(\ve{x}_k) \right| \\
& \quad \leq \lim_{K \rightarrow \infty} \lim_{N \rightarrow \infty} \lim_{T_s \rightarrow 0} \left| \hat{F}_j(\ve{x}_k) - L \widetilde{P}_K f_j(\ve{x}_k) \right| \\
& \qquad + \lim_{K \rightarrow \infty} \lim_{N \rightarrow \infty} \lim_{T_s \rightarrow 0} \left\| L \widetilde{P}_K f_j - L f_j \right\| = 0 \, .
\end{split}
\end{equation*}
This concludes the proof.
}
\end{proof}

\blue{For finite values $K$ and $N$, Theorem \ref{theo2} proves the convergence of Algorithm \ref{alg:dual} as the sampling time goes to zero, provided that the identity function is contained in the span of test functions $g_l$ (case 1). In practice, we will use Gaussian radial basis functions with $K=N$. In this case, $\widetilde{P}_K f_j$ is obtained through the interpolation of $f_j$ on a set of Gaussian radial basis functions and it can be shown that the convergence of $\widetilde{P}_K f_j$ to $f_j$ is uniform with all the derivatives (case 2) (see e.g. \cite{radial_functions_dense_Sobolev}). This basis also satisfies the independence condition \eqref{eq:cond_gl} (see e.g. \cite{Korda_convergence}).}

Table \ref{tab:comparison} summarizes the main differences between the two frameworks (main and dual methods).
\begin{table*}[h!]
		\centering
			\caption{Comparison between the main method and the dual method.}
		\label{tab:comparison}
		{\small
\begin{tabular}{|c|c|c|}
\hline
 & Main method  & Dual method \\ 
\hline
Constraints on data & $K\geq N$ & $K\leq N$ \\
Subspace of functions & Space of observables $\mathcal{F}_N$ & \guillemets{Sample space} \blue{$\widetilde{\mathcal{F}}^*_K$} \\ 
Basis functions & Monomials (i.e. $g_k=p_k$) & Preferably Gaussian RBF $g_k$ \\ 
Projected Koopman operator & $U_N f = P_N U^{T_s} f$ $\forall f \in \mathcal{F}_N$ & \blue{$\widetilde{U}_K^* f = \widetilde{P}_K^* (U^{T_s})^* \xi$ $\forall \xi \in \widetilde{\mathcal{F}}^*_K$} \\ 
Matrix representation & $\ve{\overline{U}}_N = \ve{P_x^\dagger} \, \ve{P_y} \in \mathbb{R}^{N \times N}$ & $\ve{\widetilde{U}}_K = \ve{P_y} \, \ve{P_x^\dagger}\in \mathbb{R}^{K \times K}$ \\
\hline
	\end{tabular}}
\end{table*}

\section{Illustrative examples}
\label{sec:examples}

The goal of this section is to provide several examples to illustrate the two methods, including some extensions of the main method. We do not provide here an extensive study of the performance with respect to the choice of basis functions and parameters, considering that this is out of the scope of the present paper.

We consider simulated data and, unless otherwise stated, we add a Gaussian state-dependent measurement noise with zero mean and standard deviation $\sigma_{meas} = 0.01$ (see \eqref{rel_x_y}).

\subsection{Main method}

We use the lifting method described in Section \ref{sec:lifting}, with the parameters $m=m_F=3$. We consider three systems that exhibit different types of behaviors.
\begin{enumerate}
\item{Van der Pol oscillator:} the dynamics are given by
\begin{eqnarray*}
	\dot{x}_1 & = & x_2 \\
	\dot{x}_2 & = & (1-x_1^2)x_2-\blue{x_1}
\end{eqnarray*}
and possess a stable limit cycle. 
\item{Unstable equilibrium:} the dynamics are given by
\begin{eqnarray*}
	\dot{x}_1 & = & 3 \, x_1+0.5 \, x_2 -x_1 x_2 + x_2^2 + 2 \, x_1^3 \\
	\dot{x}_2 & = & 0.5 \,x_1+4 \, x_2
\end{eqnarray*}
and are characterized by an unstable equilibrium at the origin.
\item{Chaotic Lorenz system:} the dynamics are given by
\begin{eqnarray*}
	\dot{x}_1 & = & 10(x_2-x_1) \\
	\dot{x}_2 & = & x_1(28-x_3)-x_2 \\
	\dot{x}_3 & = & x_1 x_2 - 8/3 \, x_3
\end{eqnarray*}
and exhibit a chaotic behavior.
\end{enumerate}

A set of $K$ data pairs is generated by taking snapshots at times $\{0,T_s,\dots,K/r T_s\}$ from $r$ trajectories of these systems. For the first two systems, we consider a setting that is not well-suited to a direct estimation of the derivatives: the sampling period $T_s$ is (reasonably) large and only two or three data points are taken on each trajectory. The identification of the third system, however, requires a smaller sampling period and a larger number of samples. Parameters used to generate the datasets are summarized in the left part of Table \ref{tab:examples}. 

For each model, Algorithm \ref{alg:lifting} yields the estimates $\hat{w}_k^j$ of the coefficients $w_k^j$ \blue{(Figure \ref{fig:toys_coeff})}. We compute the root mean square error
\begin{equation}
\label{eq:RMSE}
\textrm{RMSE} = \sqrt{\frac{1}{n N_F} \sum_{j=1}^n \sum_{k=1}^{N_F} \left((w_k^j) - (\hat{w}_k^j) \right)^2} 
\end{equation}
and the normalized root mean square error $\textrm{NRMSE} = \textrm{RMSE}/\overline{w}$, where $\overline{w}$ is the average value of the nonzero coefficients $|w_k^j|$. The RMSE and NRMSE values averaged over $50$ experiments are small (Table \ref{tab:examples}) and show that the lifting method achieves good performance to identify each system with a fairly low number of samples. \blue{Figure \ref{fig:toys_pred} shows predictions obtained with the identified vector field. These predictions are good, but some errors (in particular for the chaotic system) are due to measurement noise and finite-dimensional approximations. These results could be improved by increasing the number of basis functions and reducing the sampling period (not shown here). Note that this is for illustrative purposes only, since predictions could also be obtained directly through the lifted dynamics (see e.g. \cite{Korda_MPC}).}

\begin{figure}[h]
	\centering
	\includegraphics[width=0.8\linewidth]{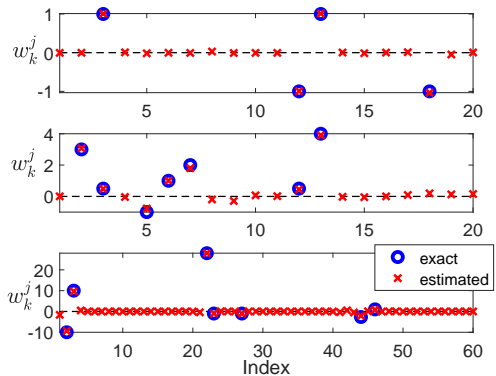}
	\caption{\blue{Vector field coefficients obtained for the Van der Pol oscillator (top), the unstable system (middle), and the chaotic Lorenz system (bottom).}}
	\label{fig:toys_coeff}
\end{figure}

\begin{table*}[t!]
		\centering
			\caption{Features of the datasets and (normalized) root mean square error averaged over $50$ simulations.}
		\label{tab:examples}
		{\small 
	\begin{tabular}{lcccccc}
		\hline 
		\hline 
		& Sampling & Total number & Number of & Initial & \multirow{2}{*}{RMSE} & \multirow{2}{*}{NRMSE} \\ 
		& period ($T_s$) &  of data pairs ($K$) & trajectories ($r$) & conditions & & \\
		\hline 
		1. Van der Pol & 0.5 & 30  & 15 & $[-1,1]^2$ & 0.023 & 0.023 \\ 
		2. Unstable & 0.2 & 20 & 20 & $[-0.5,0.5]^2$ & 0.150 & 0.087 \\ 
		3. Lorenz & 0.033 & 300 & 20 & $[-20,20]^3$ & 0.451 & 0.059 \\ 
		\hline 
		\hline 
	\end{tabular} }
\end{table*}

\begin{figure}[h]
	\centering
	\includegraphics[width=0.8\linewidth]{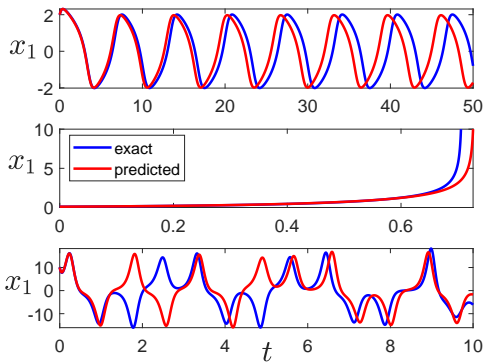}
	\caption{\blue{Prediction of trajectories using the estimated vector field (Van der Pol oscillator (top), unstable system (middle), and chaotic Lorenz system (bottom).}}
	\label{fig:toys_pred}
\end{figure}

For the three systems described above, we also consider the effect of the sampling period $T_s$ on the performance of the method (Figure \ref{fig:error_coeff}). In the noiseless case, the NRMSE decreases (exponentially) as the sampling period decreases. This is in agreement with the fact that the NRMSE tends to zero as $T_s \rightarrow 0$ (Theorem \ref{theo1}). With measurement noise, this is not the case since the method is biased. In this case, small values of the sampling period make the method more sensitive to noise, so that the minimal (nonzero) value of the NRMSE is obtained with an intermediate value of the sampling period.
\begin{figure}[h]
	\centering
	\subfigure[Van der Pol]{\includegraphics[width=0.32\linewidth]{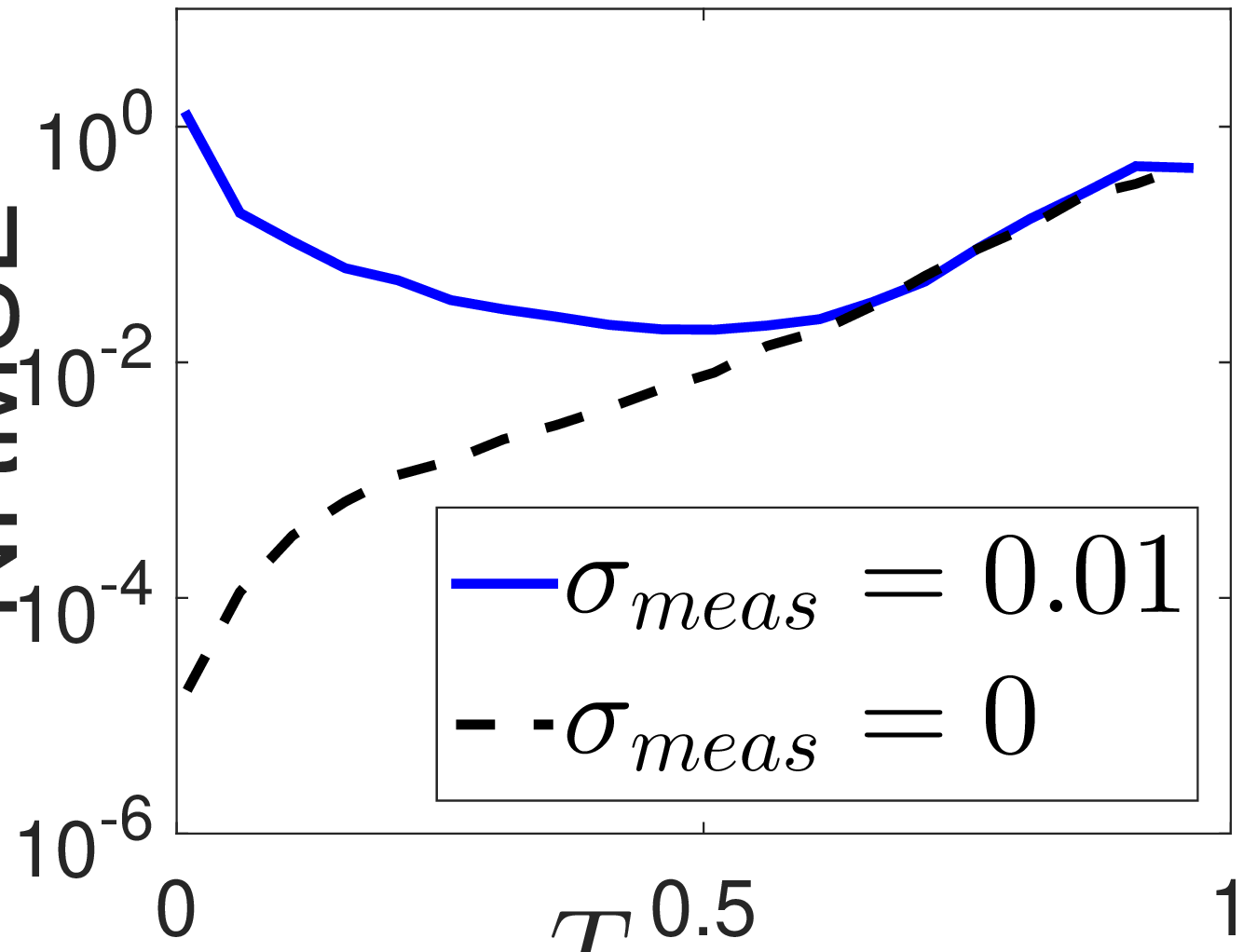}}
	\subfigure[Unstable]{\includegraphics[width=0.32\linewidth]{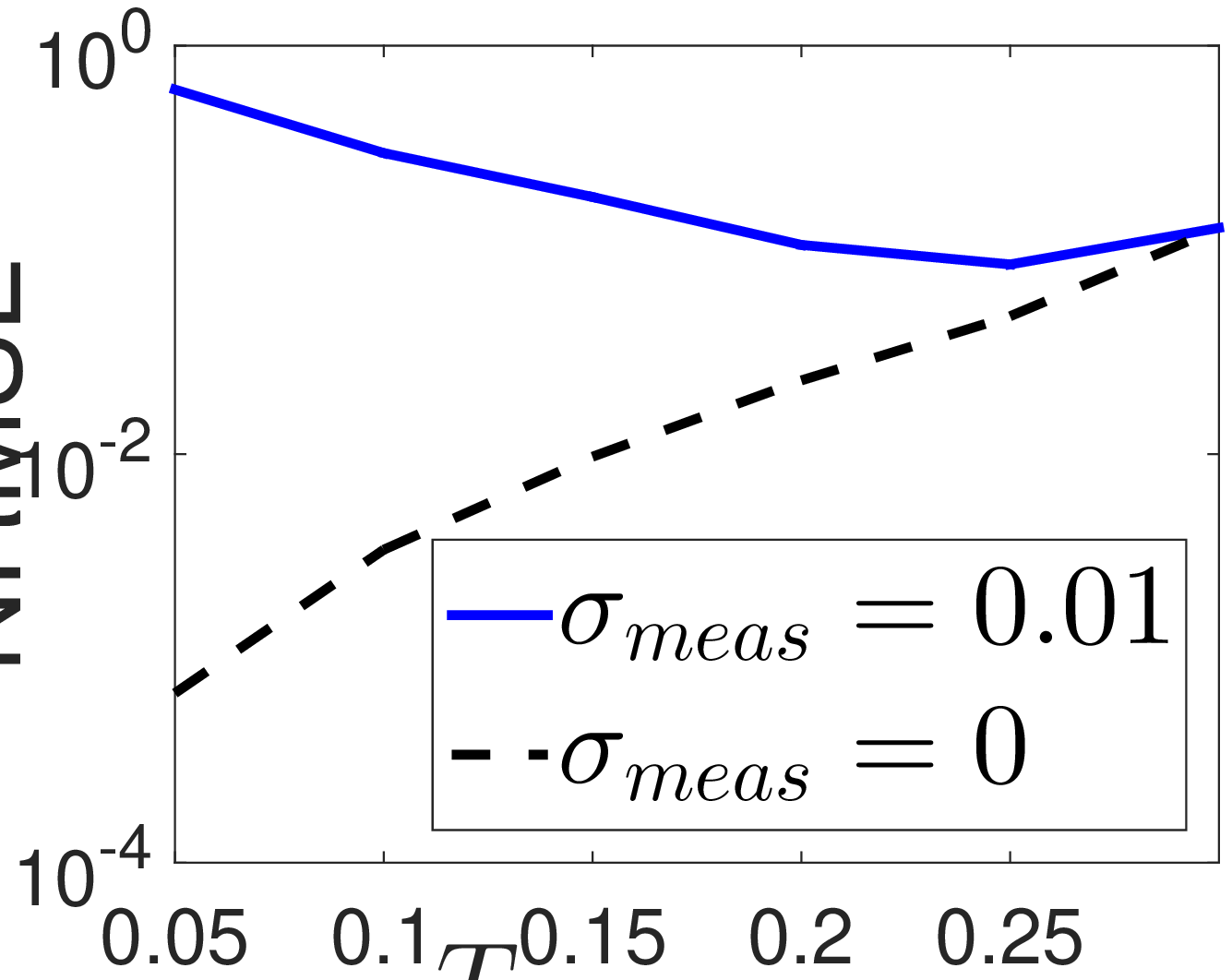}}
	\subfigure[Lorenz]{\includegraphics[width=0.32\linewidth]{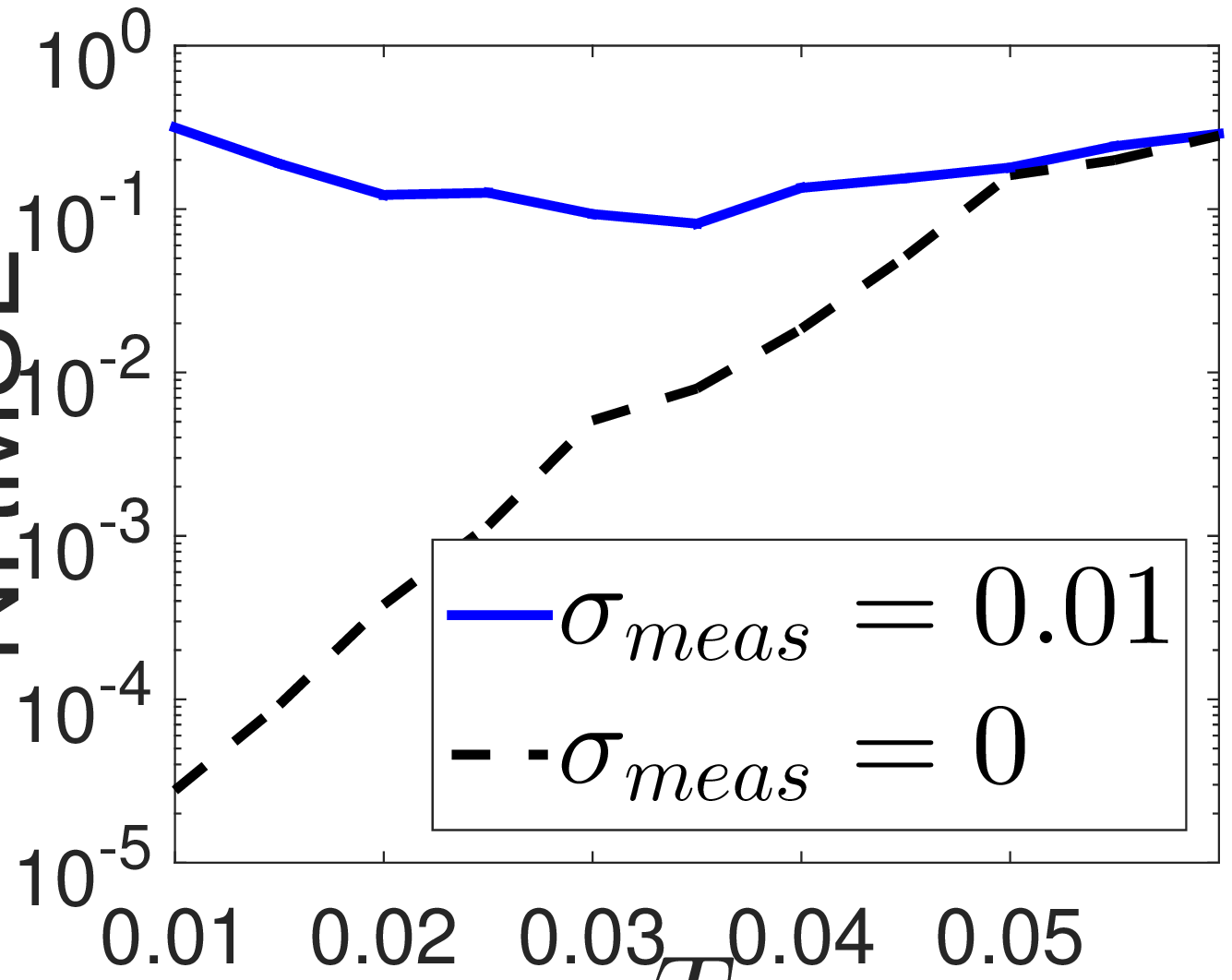}}
	\caption{Effect of the sampling period on the normalized root mean square error (averaged over $50$ experiments). Parameters are given in Table \ref{tab:examples}.}
	\label{fig:error_coeff}
\end{figure}

Next, the approximation of the vector field obtained with \eqref{eq:comput_vec_field} is compared with the approximation obtained directly from data through (central) finite differences, i.e.
\begin{equation*}
\hat{\ve{F}}(\ve{x}_k) = \frac{\ve{x}_{k+1} - \ve{x}_{k-1}}{2 T_s} \,.
\end{equation*}
We consider the three systems and compute the normalized root mean square error on the vector field
\begin{equation*}
\textrm{NRMSE}_{F} = \frac{\sqrt{\frac{1}{K-1}\sum_{j=1}^n \sum_{k=2}^{K} \left\|\hat{\ve{F}}(\ve{x}_k)-\ve{F}(\ve{x}_k)\right\|^2}}{\frac{1}{K-1} \sum_{k=2}^K \left\|\ve{F}(\ve{x}_k) \right\|}
\end{equation*}
averaged over $10$ experiments, for different values of the sampling period. The results are shown in Figure \ref{fig:comp_vec_field}. For each system, we observe that the approximation obtained with the lifting method provides an estimate with an acceptable error (e.g. $NRMSE_F < 0.1$) for larger values of the sampling period than the direct finite difference method. This approximation is also characterized by a clear transition at a critical value of the sampling period, above which the NRMSE sharply increases (not observed with the unstable system, for which the critical value is beyond the maximal integration time). These results demonstrate the need of considering an indirect method to estimate the vector field (and therefore identify the system) when the sampling period is large.

\begin{figure}[h]
	\centering
	\subfigure[Van der Pol]{\includegraphics[width=0.32\linewidth]{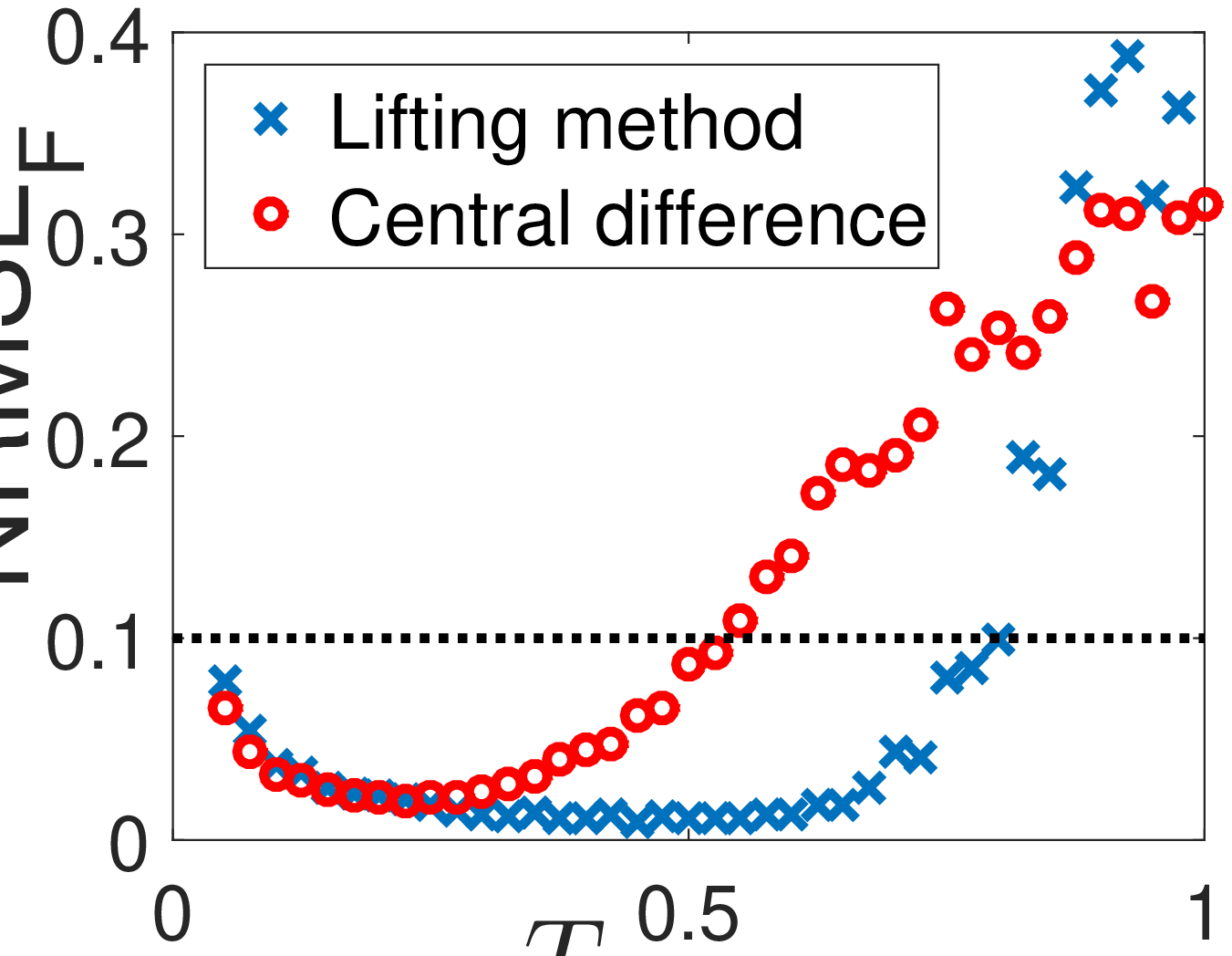}}
	\subfigure[Unstable]{\includegraphics[width=0.32\linewidth]{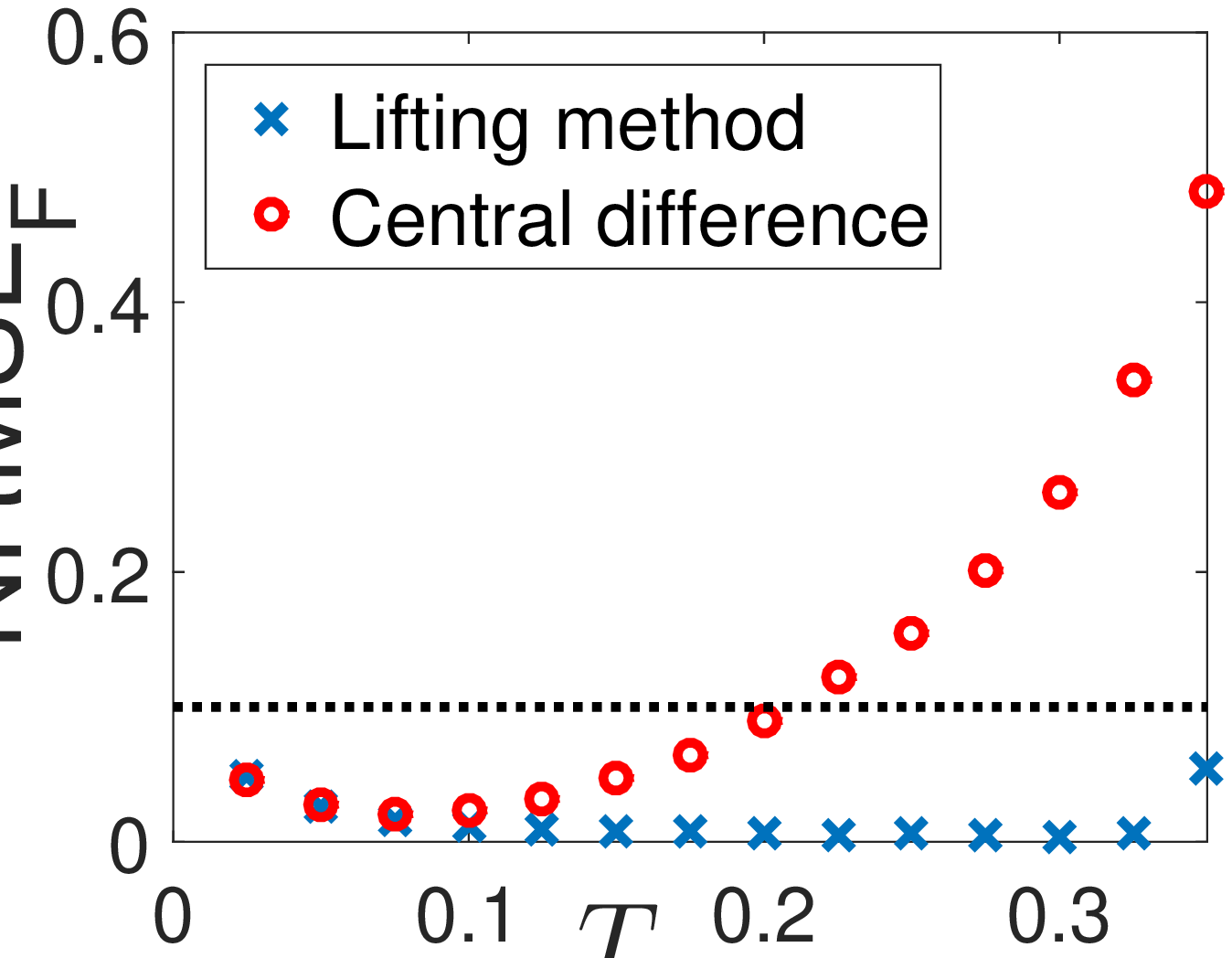}}
	\subfigure[Lorenz]{\includegraphics[width=0.32\linewidth]{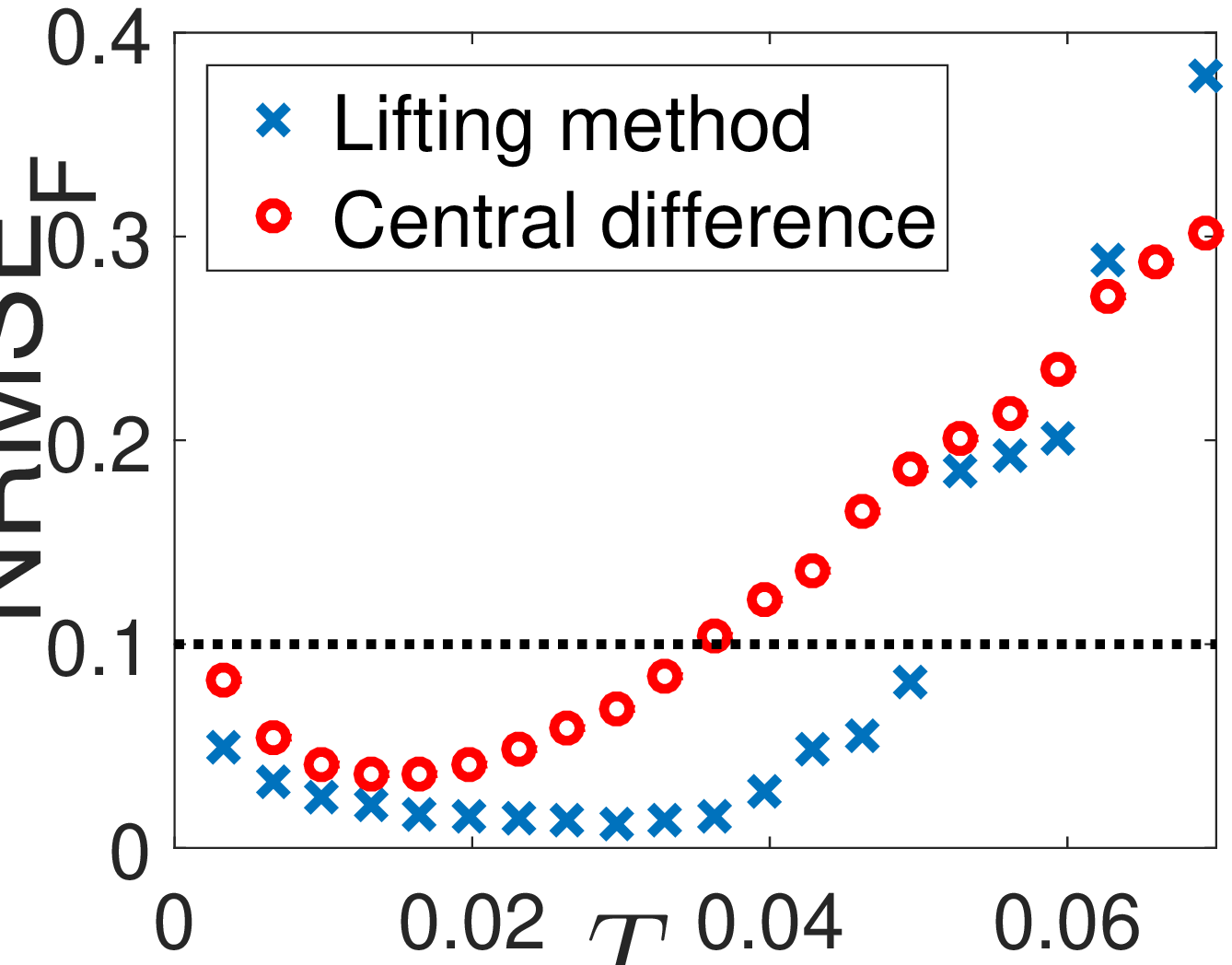}}
	\caption{Comparison of the normalized root mean square error on the vector field estimated with the lifting method and with a finite difference method (averaged over $10$ experiments). The parameters are the same as in Table \ref{tab:examples}, except for the unstable system where $K=40$ ($2$ data pairs on each trajectory) and the initial conditions are in the set $[-0.1,0.1]^2$.}
	\label{fig:comp_vec_field}
\end{figure}

\subsection{Extensions}
\label{sec:example_extension}

We now illustrate several extensions of the lifting method mentioned in Section \ref{sec:extensions}: systems with inputs, process noise, and non-polynomial vector fields.

\subsubsection{Input and process noise} We consider the forced Duffing system
\begin{eqnarray}
\label{Duff1}
\dot{x}_1 & = & x_2 \\
\label{Duff2}
\dot{x}_2 & = & x_1-x_1^3-0.2 \, x_2 + 0.2 \, x_1^2 \, \cos(t)
\end{eqnarray}
and generate $K=250$ snapshot data pairs from $5$ trajectories ($50$ on each), with initial conditions on $[-1,1]^2$. The lifting method provides a good estimation of the vector field (including the forcing term $0.2 \, x_1^2 \, \cos(t)$). The RMSE (see Equation \eqref{eq:RMSE}) and NRMSE computed over all coefficients (including those related to the forcing term) are given in Table \ref{tab:examples_input} for different values of the sampling period. Note that we use again the parameters $m=m_F=3$.

\begin{table}[h]
	\centering
		\caption{(Normalized) root mean square error (averaged over $50$ experiments) related to the identification of the forced Duffing system.}
		\label{tab:examples_input}
	\begin{tabular}{ccc}
		\hline 
		\hline 
		Sampling  & \multirow{2}{*}{RMSE} & \multirow{2}{*}{NRMSE} \\ 
		period ($T_s$) & & \\
		\hline 
		0.2 & 0.032 & 0.046 \\
		0.4 & 0.031 & 0.045 \\
		0.6 & 0.057 & 0.084 \\
		\hline 
		\hline 
	\end{tabular} 
\end{table}

Now, we replace the forcing term in \eqref{Duff2} by the white noise $\eta(t)$ with different values of the standard deviation $\sigma_{proc}$ (note that we still add measurement noise with $\sigma_{meas}=0.01$). We generate $K=500$ snapshot data pairs from $10$ trajectories computed with the Euler-Maruyama scheme, with initial conditions on $[-1,1]^2$. The sampling period is equal to $T_s = 0.2$. As shown in Table \ref{tab:examples_noise}, the error is small even with strong process noise, suggesting that the method is robust against process noise. 

\begin{table}[h]
	\centering
		\caption{(Normalized) root mean square error (averaged over $10$ experiments) related to the identification of the Duffing system with process noise.}
		\label{tab:examples_noise}
	\begin{tabular}{ccc}
		\hline 
		\hline 
		Noise strength  & \multirow{2}{*}{RMSE} & \multirow{2}{*}{NRMSE} \\ 
		 ($\sigma_{proc}$) & & \\
		\hline 
		0.2 & 0.063 & 0.079 \\
		0.4 & 0.065 & 0.082 \\
		0.6 & 0.074 & 0.092 \\
		0.8 & 0.067 & 0.084 \\
		0.1 & 0.094 & 0.117 \\
		\hline 
		\hline 
	\end{tabular} 
\end{table}

\subsubsection{Non polynomial vector fields} In this example, we consider a genetic toggle switch (see e.g. \cite{Gardner})
\begin{eqnarray*}
\dot{x}_1 & = & - x_1 + 2 \, x_2  \\
\dot{x}_2 & = & - x_2 + \frac{2}{1+x_3^2} \\
\dot{x}_3 & = & - 2 \, x_3 + 2 \, x_4 \\
\dot{x}_4 & = & - 2 \, x_4 + \frac{1}{1+x_1^3}
\end{eqnarray*}
and we generate $K=50$ snapshot data pairs from $50$ trajectories, with initial conditions on $[0,1]^4$. The sampling period is $T_s=0.1$. Since the vector field is not polynomial, we use the extension presented in Section \ref{sec:other_vec}. The basis functions are the $5$ monomials of total degree $0$ and $1$ (i.e. $m=1$), to which we add $12$ Hill functions 
\begin{equation}
\label{eq:Hill}
\frac{1}{1+x_k^l} \quad k=\{1,2,3,4\} \,, \quad l=\{1,2,3\} \,.
\end{equation}
When there is no measurement noise, all coefficients (including those related to non-polynomial terms) are inferred correctly and we obtain a NRMSE equal to $0.008$ (averaged over $50$ experiments). However, the results are sensitive to noise in this case. With a measurement noise with $\sigma_{meas}=0.001$, the NRMSE increases to $0.494$. As shown below, the dual method is more robust to noise in this case.

\subsection{Dual method}
\label{sec_example_dual}

We illustrate the dual method in the case of a non-polynomial vector field. The main interest of the method, however, is its use with high-dimensional datasets, where the number of basis functions $N$ is (much) larger than the number of sample points $K$. This will be illustrated in the next section.

The dual method requires to solve a regression problem. When $K < N_F$, we solve the (underconstrained) Lasso problem \eqref{eq:Lasso} with the MATLAB toolbox ``yall1'' \cite{Yall1, Yall1_manual} (L1-L2 problem, with the parameter $\rho=0.01$). When $K \geq N_F$, we solve the (overconstrained) problem \eqref{eq:Lasso} with the MATLAB function ``lasso'' (with the parameter $\lambda=1/K$). Note that the value of the regularization parameter might not be optimal in all cases, and we did not extensively study its effect on the performance of the algorithm. In the following, we only use Gaussian radial basis functions with $\gamma=0.1$ or $\gamma=0.01$. Numerical simulations performed with monomial bases (not shown here) yield similar results for small dimensions, but are less accurate and more computationally expensive for large dimensions.

We consider the toggle switch system introduced in Section \ref{sec:example_extension}. Sample points are generated in the same conditions (i.e. $K=50$, $T_s=0.1$). We consider Gaussian radial basis functions with $\gamma=0.1$ and $17$ library functions ($5$ monomials of total degree $0$ and $1$, and $12$ Hill functions \eqref{eq:Hill}). With no noise, the NRMSE (averaged over $50$ experiments) is equal to $0.064$, which is worse than with the main method (Section \ref{sec:example_extension}). However, we obtain a NRMSE equal $0.117$ with $\sigma_{meas} = 0.001$ and equal to $0.637$ with $\sigma_{meas} = 0.01$. This shows that, in this case, the dual method is more robust to measurement noise than the main method. \blue{This might be due to the sparsity constraint that we impose in the dual method}.

\subsection{Application to network identification}

In the context of dynamical systems, each state can be seen as the node of a network. Moreover, a link can be drawn from node $i$ to node $j$ if the dynamics of the state $x_j$ depends on the state $x_i$. Under the assumption that the vector field is of the form \eqref{eq:vec_field}, there is a link from node $i$ to node $j$ if there is at least one nonzero coefficient $w_k^j$ such that the corresponding library function $h_k$ depends on $x_i$.

Network reconstruction aims at predicting links between states from data, a goal which is equivalent to finding nonzero coefficients $w_k^j$ in our setting. We will consider that estimated coefficients $\hat{w}_k^j$ with a small absolute value are mainly due to measurement noise and have an exact value $w_k^j$ equal to zero. Hence, we decide that a link is present in the network only if the related value $|w_k^j|$ is above a given threshold. To evaluate the performance of the method, one can compute the true positive rate (i.e. number of correctly identified links divided by the actual number of links) and the false positive rate (i.e. number of incorrectly identified links divided by the actual number of missing links). Varying the threshold value, we can plot the true positive rate against the false positive rate, which corresponds to the receiver operating characteristic (ROC) curve. If the area under the ROC curve (AUROC) is close to one, the network inference method provides good results (bad result correspond to a value close to $0.5$).

\subsubsection{Kuramoto oscillators} We consider a network of $n$ Kuramoto phase oscillators
\begin{equation*}
\theta_{i} = \omega_i + \frac{C}{n} \sum_{j=1}^n a_{ij} \sin(\theta_j-\theta_i) \qquad i=1,\dots,n
\end{equation*}
with $\theta_i \in [0,2\pi)$. The coupling strength is set to $C=10$ and the natural frequencies $\omega_i$ are uniformly randomly distributed on $[0,0.1]$. The values $a_{ij}$ are the entries of the weighted adjacency matrix of a random Erd{\H o}s-R\'enyi graph (with a probability $p_{link}=0.3$ for any two nodes to be connected). The link weights are uniformly randomly distributed on $[0,1]$.

For two networks ($n=20$ and $n=100$), we generate $K$ sample pairs from $K/5$ trajectories ($5$ data pairs on each trajectory), with $T_s=0.2$. Initial conditions are uniformly distributed on $[0,2\pi)^n$. Note that we do not consider data points on $[0,2\pi)$ but on the real line $\mathbb{R}$ (i.e. without the modulo operation) where there is no discontinuity between $0$ and $2\pi$. We use the dual method with Gaussian radial basis functions (with $\gamma=0.1$) and with $N_F = n$ library functions
\begin{equation*}
\begin{split}
\{ 1 \, , \, & \sin(\theta_1-\theta_i) \, , \, \dots \, , \, \sin(\theta_{i-1}-\theta_i) \, , \\
& \sin(\theta_{i+1}-\theta_i) \, , \, \dots \, , \, \sin(\theta_{n}-\theta_i) \}
\end{split}
\end{equation*}
for the $i$th component of the vector field. ROC curves are shown in Figure \ref{fig:ROC_Kuramoto} and the results are summarized in Table \ref{tab:examples_Kuramoto}, for different values of $K$ and $\sigma_{meas}$. They show that the dual method achieves good performance to reconstruct the whole network. In particular, with high threshold values, one can infer many true positive links with no false positive links.

\begin{figure}[h]
	\centering
	\includegraphics[width=0.5\linewidth]{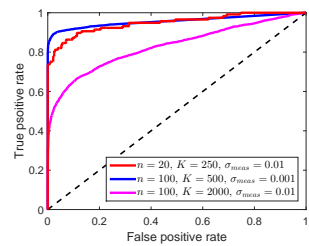}
	\caption{ROC curves obtained with the dual method for the reconstruction of a network of Kuramoto oscillators.}
	\label{fig:ROC_Kuramoto}
\end{figure}

 \begin{table}[h]
 	\centering
 		\caption{Results obtained with the dual method for the reconstruction of a network of Kuramoto oscillators.}
 	 	\label{tab:examples_Kuramoto}
 	\begin{tabular}{cccc}
 		\hline 
 		\hline 
 		$n$  & $K$ & $\sigma_{meas}$ & AUROC \\ 
 		\hline 
 		20 & 250 & 0.01 & 0.95 \\
 		100 & 500 & 0.001 & 0.96 \\
 		100 & 2000 & 0.01 & 0.83 \\
 		\hline 
 		\hline 
 	\end{tabular} 
 \end{table}

\subsubsection{Network with nonlinear couplings} We consider a network where each state is directly influenced by other states through $n_{inter}$ quadratic and cubic nonlinearities. Each nonlinear interaction depends on at most two states. \blue{The dynamics of the system are given by
\begin{equation}
\label{dyn_network}
\dot{x}_j = w_1^j x_j  + \sum_{k=2}^{N_F} w_k^j \, h_k \qquad j=1,\dots,n
\end{equation}
where the functions $h_k$ are monomials of total degree less or equal to $3$. For each $j$, only $n_{inter}$ coefficients $w_k^j$ are nonzero and associated with monomials of the form $x_k^p x_l^q$, with total degree $p+q \in \{2,3\}$. The coefficients $w_1^j$ are chosen according to a uniform distribution on $[0,1]$ and the coefficients $w_k^j$, with $k>1$, are distributed according to a Gaussian distribution of zero mean and standard deviation equal to one.} The first term in \eqref{dyn_network} is a linear term that ensures local stability of the origin. For several network sizes ($n\in\{20,50,100\}$), we generate $K$ samples from $K/2$ trajectories ($2$ data pairs on each trajectory), with $T_s=0.5$. Initial conditions are uniformly randomly distributed on $[-0.5,0.5]^n$. 

Although we could also consider the main method for small networks (typically $n \leq 20$), we use only the dual method with Gaussian radial basis functions (with $\gamma=0.01$). The library functions are monomials of total degree less or equal to $3$. The method provides an accurate estimation of the vector field and a good reconstruction of the network (Table \ref{tab:examples_nonlinear_net}). The ROC curves depicted in Figure \ref{fig:examples_nonlinear_net}(a) show that most of half of the links can be inferred with no false positive link (with high threshold values). As shown in Figure \ref{fig:examples_nonlinear_net}(b-d), the method is also efficient to infer the nature of the interactions (e.g. quadratic, cubic). Taking advantage of sparsity, it uses not more than $1000$ sample points to identify up to $17.10^6$ coefficients (most of which are zero). We finally note that, for larger networks, the use of monomials as library functions becomes too demanding in terms of memory. In this case, the dual method can still be used to estimate the value of the vector field at the sample points, but should be combined with other (regression) methods to infer the network.

 \begin{table}[h!]
	\centering
		\caption{Results obtained with the dual method for the reconstruction of a network with quadratic and cubic interactions.}
		\label{tab:examples_nonlinear_net}
	\begin{tabular}{ccccc}
		\hline 
		\hline 
		$n$  & $n_{inter}$ &  $K$ & AUROC & NRMSE \\ 
		\hline 
		20 & 5 & 200 & 0.94 & 0.019 \\
		50 & 15 & 600 & 0.87 & 0.015 \\
		100 & 10 & 1000 & 0.91 & 0.004 \\
		\hline 
		\hline 
	\end{tabular} 
\end{table}

\begin{figure}[h!]
	\centering
	\subfigure[ROC curves]{\includegraphics[width=0.45\linewidth]{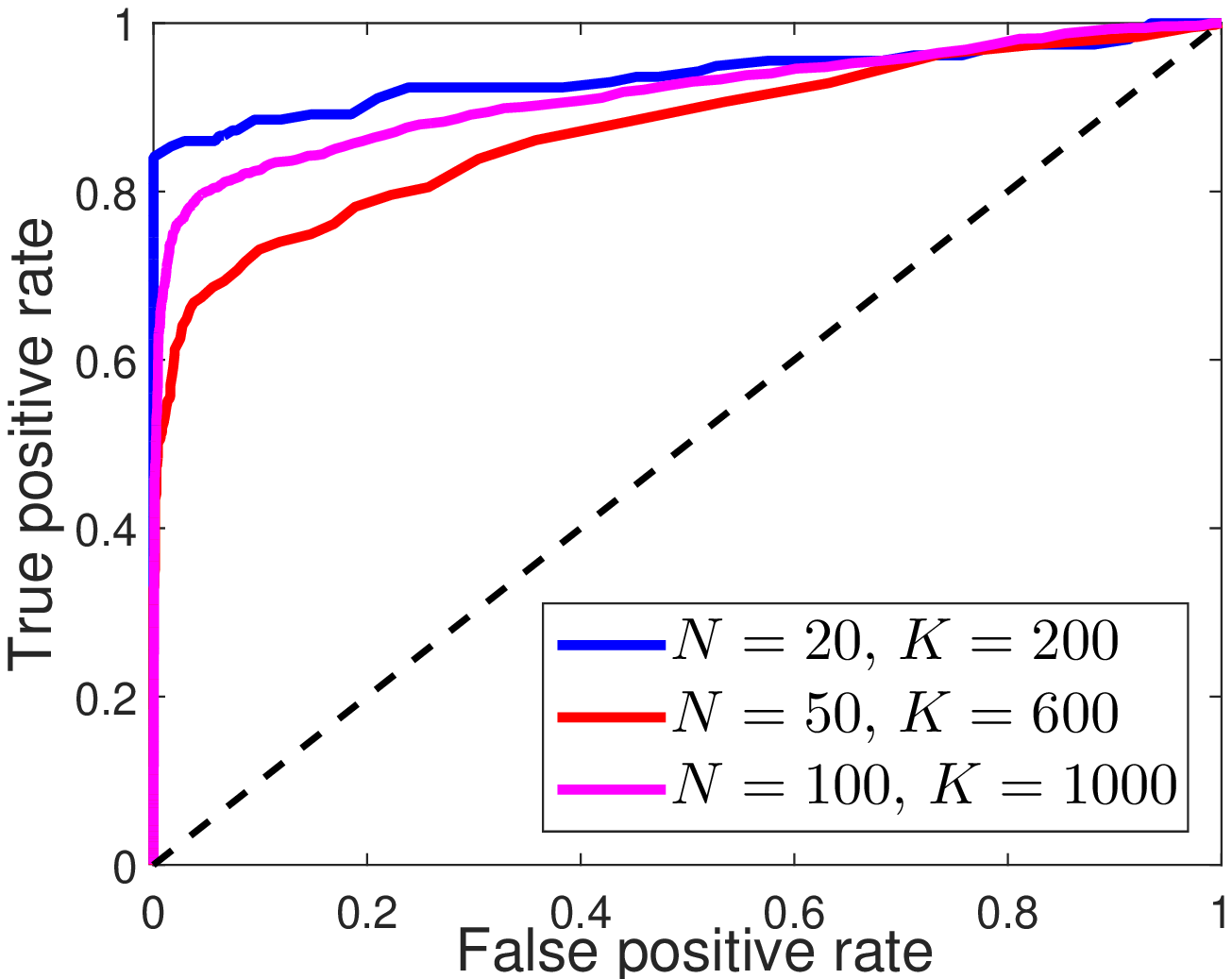}}
	\subfigure[$N=20$, $K=200$]{\includegraphics[width=0.45\linewidth]{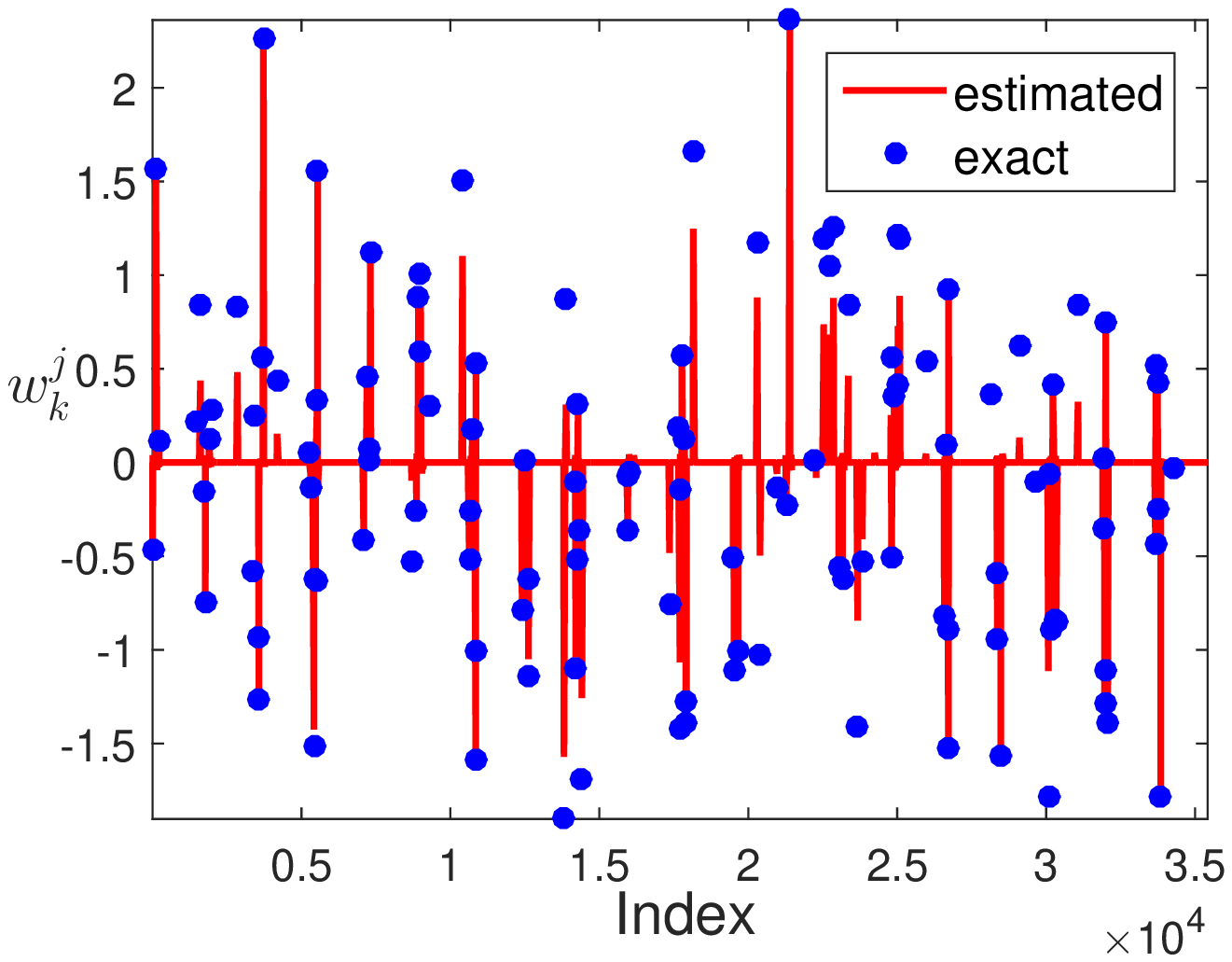}}
	\subfigure[$N=50$, $K=600$]{\includegraphics[width=0.45\linewidth]{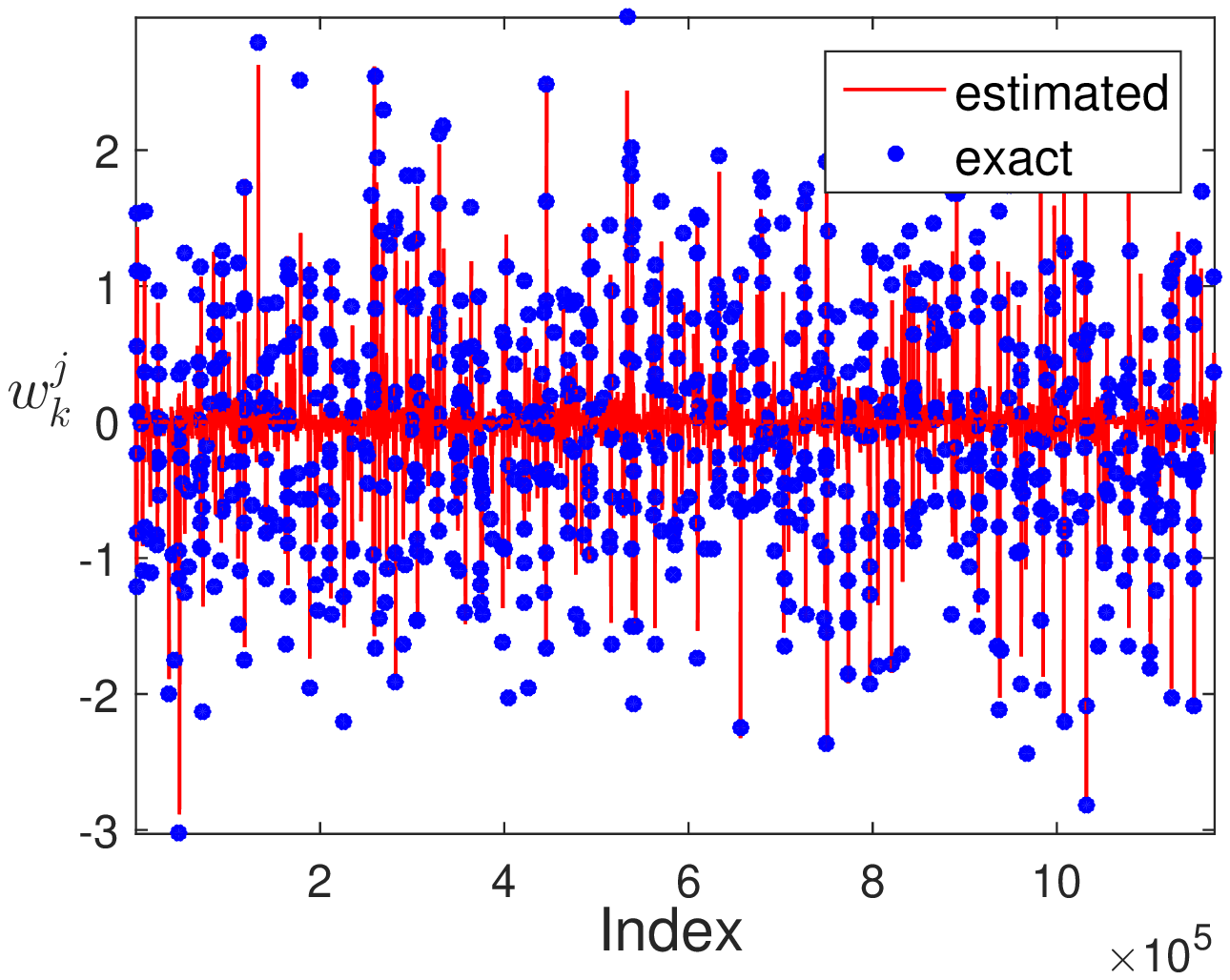}}
	\subfigure[$N=100$, $K=1000$]{\includegraphics[width=0.45\linewidth]{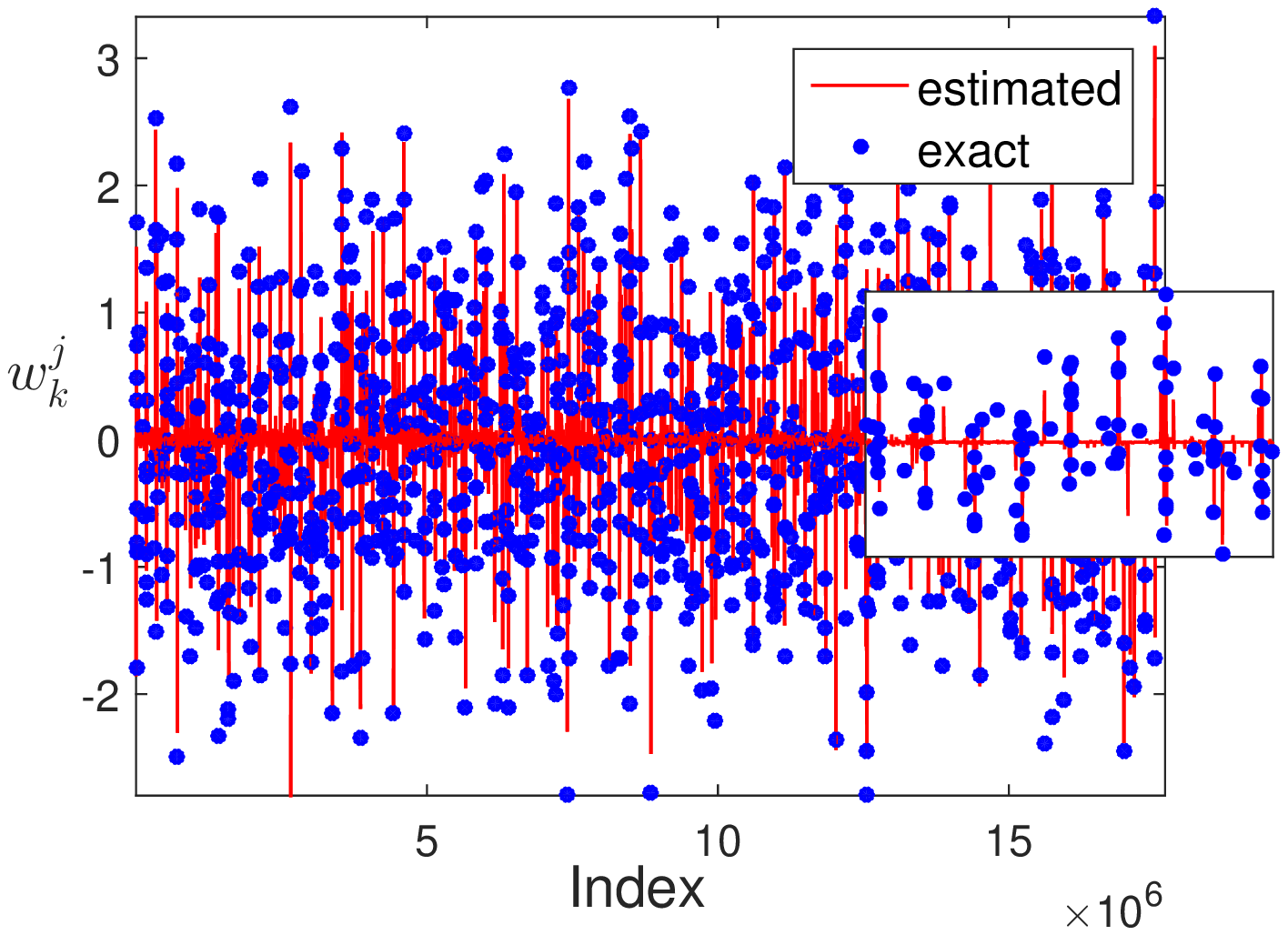}}
	\caption{ROC curves and vector field coefficients obtained with the dual method for the reconstruction of a network with quadratic and cubic interactions. In (d), the inset shows a close-up of some estimated coefficients.}
	\label{fig:examples_nonlinear_net}
\end{figure}

\section{Conclusion}
\label{sec:conclu}

We have proposed a novel method for nonlinear systems identification. This method relies on a lifting technique developed in an operator-theoretic framework: it aims at identifying the linear Koopman operator in the space of observables. Key advantages of the method are that numerical schemes rely only on linear techniques and do not require the estimation of state time derivatives. For these reasons, this is a promising alternative to direct identification methods. As shown with several examples, the method is efficient to recover the vector field of several classes of systems, even from small time series with low sampling rate. Moreover, a dual method is also proposed to identify high-dimensional systems and is successfully applied to network reconstruction. Theoretical results also prove the convergence of the two methods in optimal conditions.

The results presented in this paper open the door to further developments and improvements of lifting techniques for nonlinear systems identification, some of which are related to recent advances in Koopman operator theory. For instance, identification lifting techniques with dictionary learning could be developed \cite{Kevrekidis_learning}. Extensions to general vector fields might also be considered, possibly without using library functions. Toward this end, lifting techniques could be combined with other methods: identify unknown parameters with Kalman filtering \cite{Wei_Kalman}, consider rational functions in the vector field with alternating directions method \cite{Brunton_network_inference}, apply machine learning regression techniques on time derivatives estimated with the dual method, etc. Moreover, we might improve the method robustness to (measurement) noise and provide numerical schemes that are unbiased and consistent. In this context, Bayesian inference could be considered as a relevant approach. A careful study of the matrix logarithm used in the lifting method could also help to select the good branch (instead of the principal one), a strategy which might improve the performances when the sampling rate is low. Theoretical results could also be obtained to provide bounds on the estimation error. \blue{Finally, a potential extension of the proposed approach is to consider the case of unobserved states (e.g. hidden nodes in the context of network identification). In this context, classic linear identification methods could be exploited (e.g. subspace identification methods \cite{subspace_identif}).  In the same line, connections with (nonlinear) system identification methods such as the modulating function approach \cite{modulating_functions} could be investigated.}

\section*{Acknowledgments}
 
The authors acknowledge J. Winkin and F. Lamoline for fruitful discussions and suggestions, and for their help in the proofs presented in the manuscript. They also wish to thank S. Brunton and N. Kutz for suggesting Kuramoto oscillators in the reconstruction problem. The authors acknowledge support from the Luxembourg National Research Fund. This paper presents research results of the Belgian Network DYSCO (Dynamical Systems, Control, and Optimization), funded by the Interuniversity Attraction Poles Programme initiated by the Belgian Science Policy Office. This research used resources of the \guillemets{Plateforme Technologique de Calcul Intensif (PTCI)} located at the University of Namur, Belgium, which is supported 
by the F.R.S.-FNRS under the convention No.2.5020.11. The PTCI is member of the \guillemets{Consortium des Équipements de Calcul Intensif (CÉCI)}.

\bibliographystyle{IEEEtran}

\begin{thebibliography}{10}
\providecommand{\url}[1]{#1}
\csname url@samestyle\endcsname
\providecommand{\newblock}{\relax}
\providecommand{\bibinfo}[2]{#2}
\providecommand{\BIBentrySTDinterwordspacing}{\spaceskip=0pt\relax}
\providecommand{\BIBentryALTinterwordstretchfactor}{4}
\providecommand{\BIBentryALTinterwordspacing}{\spaceskip=\fontdimen2\font plus
\BIBentryALTinterwordstretchfactor\fontdimen3\font minus
  \fontdimen4\font\relax}
\providecommand{\BIBforeignlanguage}[2]{{%
\expandafter\ifx\csname l@#1\endcsname\relax
\typeout{** WARNING: IEEEtran.bst: No hyphenation pattern has been}%
\typeout{** loaded for the language `#1'. Using the pattern for}%
\typeout{** the default language instead.}%
\else
\language=\csname l@#1\endcsname
\fi
#2}}
\providecommand{\BIBdecl}{\relax}
\BIBdecl


\bibitem{nonlin_identif_Wiener}
N.~Wiener, ``Nonlinear problems in random theory,'' \emph{Nonlinear Problems in
  Random Theory, by Norbert Wiener, pp. 142. ISBN 0-262-73012-X. Cambridge,
  Massachusetts, USA: The MIT Press, August 1966.(Paper)}, p. 142, 1966.

\bibitem{nonlin_identif_narx}
I.~Leontaritis and S.~A. Billings, ``{Input-output parametric models for
  non-linear systems part I: deterministic non-linear systems},''
  \emph{International journal of control}, vol.~41, no.~2, pp. 303--328, 1985.

\bibitem{nonlin_identif_neural}
K.~S. Narendra and K.~Parthasarathy, ``Identification and control of dynamical
  systems using neural networks,'' \emph{IEEE Transactions on neural networks},
  vol.~1, no.~1, pp. 4--27, 1990.

\bibitem{nonlin_identif_survey1}
J.~Sj{\"o}berg, Q.~Zhang, L.~Ljung, A.~Benveniste, B.~Delyon, P.-Y. Glorennec,
  H.~Hjalmarsson, and A.~Juditsky, ``Nonlinear black-box modeling in system
  identification: a unified overview,'' \emph{Automatica}, vol.~31, no.~12, pp.
  1691--1724, 1995.

\bibitem{nonlin_identif_survey2}
R.~Haber and H.~Unbehauen, ``Structure identification of nonlinear dynamic
  systems—a survey on input/output approaches,'' \emph{Automatica}, vol.~26,
  no.~4, pp. 651--677, 1990.

\bibitem{Ljung}
L.~Ljung, ``System identification,'' in \emph{Signal analysis and
  prediction}.\hskip 1em plus 0.5em minus 0.4em\relax Springer, 1998, pp.
  163--173.

\bibitem{nonlin_estim_book}
Y.~Bard, \emph{Nonlinear parameter estimation}.\hskip 1em plus 0.5em minus
  0.4em\relax Academic press, 1974.

\bibitem{Aastrom_identif}
K.~J. {\AA}str{\"o}m and P.~Eykhoff, ``System identification—a survey,''
  \emph{Automatica}, vol.~7, no.~2, pp. 123--162, 1971.

\bibitem{param_estim_der_estim}
J.~M. Varah, ``A spline least squares method for numerical parameter estimation
  in differential equations,'' \emph{SIAM Journal on Scientific and Statistical
  Computing}, vol.~3, no.~1, pp. 28--46, 1982.

\bibitem{Wei}
W.~Pan, Y.~Yuan, J.~Goncalves, and G.-B. Stan, ``A sparse {B}ayesian approach
  to the identification of nonlinear state-space systems,'' \emph{IEEE
  Transactions On Automatic Control}, vol.~61, no.~1, pp. 182--187, 2016.

\bibitem{Brunton}
S.~L. Brunton, L.~P. Proctor, and J.~N. Kutz, ``Discovering governing equations
  from data by sparse identification of nonlinear dynamical systems,''
  \emph{Proceedings of the National Academy of Sciences}, vol. 113, no.~15, pp.
  3932--3937, 2016.

\bibitem{param_estim_nonlin_lsq2}
H.~G. Bock, ``Recent advances in parameter identification techniques for ode,''
  in \emph{Numerical treatment of inverse problems in differential and integral
  equations}.\hskip 1em plus 0.5em minus 0.4em\relax Springer, 1983, pp.
  95--121.

\bibitem{Budisic_Koopman}
M.~Budi{\v{s}}i{\'c}, R.~Mohr, and I.~Mezi{\'c}, ``{Applied Koopmanism},''
  \emph{Chaos}, vol.~22, no.~4, pp. 047\,510--047\,510, 2012.

\bibitem{Koopman}
B.~O. Koopman, ``{Hamiltonian systems and transformation in Hilbert space},''
  \emph{Proceedings of the National Academy of Sciences of the United States of
  America}, vol.~17, no.~5, p. 315, 1931.

\bibitem{Mezic}
I.~Mezi{\'c}, ``Spectral properties of dynamical systems, model reduction and
  decompositions,'' \emph{Nonlinear Dynamics}, vol.~41, no. 1-3, pp. 309--325,
  2005.

\bibitem{Mauroy_Mezic_stability}
A.~Mauroy and I.~Mezi{\'c}, ``{Global stability analysis using the
  eigenfunctions of the Koopman operator},'' \emph{IEEE Transactions On
  Automatic Control}, vol.~61, no.~3, pp. 3356--3369, 2016.

\bibitem{Lan}
Y.~Lan and I.~Mezi{\'c}, ``{Linearization in the large of nonlinear systems and
  Koopman operator spectrum},'' \emph{Physica D}, vol. 242, pp. 42--53, 2013.

\bibitem{Sootla_Mauroy_basins}
A.~Mauroy and A.~Sootla, ``Geometric properties of isostables and basins of
  attraction of monotone systems,'' 2017, to appear in IEEE Transactions on
  Automatic Control.

\bibitem{Muller_delay_Koopman}
D.~M{\"u}ller, A.~Otto, and G.~Radons, ``{From dynamical systems with
  time-varying delay via circle maps to Koopmanism},'' 2017, arXiv preprint
  arXiv:1701.05136.

\bibitem{Arbabi}
H.~Arbabi and I.~Mezi{\'c}, ``{Ergodic theory, Dynamic Mode Decomposition and
  computation of spectral properties of the Koopman operator},'' \emph{SIAM
  Journal on Applied Dynamical Systems}, vol.~16, no.~4, pp. 2096--2126, 2017.

\bibitem{Rowley}
C.~W. Rowley, I.~Mezi{\'c}, S.~Bagheri, P.~Schlatter, and D.~S. Henningson,
  ``Spectral analysis of nonlinear flows,'' \emph{Journal of Fluid Mechanics},
  vol. 641, pp. 115--127, 2009.

\bibitem{Schmid}
P.~J. Schmid, ``Dynamic mode decomposition of numerical and experimental
  data,'' \emph{Journal of Fluid Mechanics}, vol. 656, pp. 5--28, 2010.

\bibitem{Tu}
J.~H. Tu, C.~W. Rowley, D.~M. Luchtenburg, S.~L. Brunton, and J.~N. Kutz, ``{On
  dynamic mode decomposition: Theory and applications},'' \emph{Journal of
  Computational Dynamics}, vol.~1, no.~2, pp. 391 -- 421, December 2014.

\bibitem{Surana}
A.~Surana and A.~Banaszuk, ``{Linear observer synthesis for nonlinear systems
  using Koopman operator framework},'' in \emph{Proceedings of the IFAC
  conference}, vol.~49, no.~18.\hskip 1em plus 0.5em minus 0.4em\relax
  Elsevier, 2016, pp. 716--723.

\bibitem{Korda_MPC}
M.~Korda and I.~Mezi{\'c}, ``Linear predictors for nonlinear dynamical systems:
  Koopman operator meets model predictive control,'' \emph{Automatica},
  vol.~93, pp. 149--160, 2018.

\bibitem{Kaiser_eigenfunction_control}
E.~Kaiser, J.~N. Kutz, and S.~L. Brunton, ``Data-driven discovery of koopman
  eigenfunctions for control,'' 2017, arXiv preprint arXiv:1707.01146.

\bibitem{Susuki}
Y.~Susuki and I.~Mezi{\'c}, ``{Nonlinear Koopman modes and power system
  stability assessment without models},'' \emph{IEEE Transactions On Power
  Systems}, vol.~29, no.~2, pp. 899--907, March 2014.

\bibitem{Rowley_EDMD}
M.~O. Williams, I.~G. Kevrekidis, and C.~W. Rowley,
  ``\BIBforeignlanguage{English}{{A data-driven approximation of the Koopman
  operator: Extending dynamic mode decomposition}},''
  \emph{\BIBforeignlanguage{English}{Journal of Nonlinear Science}}, vol.~25,
  no.~6, pp. 1307--1346, 2015.

\bibitem{Mauroy_Goncalves_CDC}
A.~Mauroy and J.~Goncalves, ``{Linear identification of nonlinear systems: A
  lifting technique based on the Koopman operator},'' in \emph{Proceedings of
  the 55th IEEE Conference on Decision and Control}, 2016, pp. 6500--6505.

\bibitem{Riseth}
A.~N. Riseth and J.~P. Taylor-King, ``Operator fitting for parameter estimation
  of stochastic differential equations,'' \emph{arXiv preprint
  arXiv:1709.05153}, 2017.

\bibitem{Williams_kernel}
M.~O. Williams, C.~W. Rowley, and I.~G. Kevrekidis, ``{A kernel-based approach
  to data-driven Koopman spectral analysis},'' \emph{Journal of Computational
  Dynamics}, vol.~2, no.~2, pp. 247--265, 2015.

\bibitem{lifting_robot}
D.~Bruder, C.~D. Remy, and R.~Vasudevan, ``{Nonlinear system identification of
  soft robot dynamics using Koopman operator theory},'' \emph{arXiv preprint
  arXiv:1810.06637}, 2018.

\bibitem{Lasota_book}
A.~Lasota and M.~C. Mackey, \emph{Chaos, Fractals, and Noise: Stochastic
  aspects of dynamics}.\hskip 1em plus 0.5em minus 0.4em\relax Springer-Verlag,
  1994.

\bibitem{Kazantzis}
N.~Kazantzis and C.~Kravaris, ``{Time-discretization of nonlinear control
  systems via Taylor methods},'' \emph{Computers \& chemical engineering},
  vol.~23, no.~6, pp. 763--784, 1999.

\bibitem{Engel_Nagel}
K.-J. Engel and R.~Nagel, \emph{One-parameter semigroups for linear evolution
  equations}.\hskip 1em plus 0.5em minus 0.4em\relax Springer Science \&
  Business Media, 1999, vol. 194.

\bibitem{Korda_convergence}
M.~Korda and I.~Mezi{\'c}, ``{On convergence of extended dynamic mode
  decomposition to the Koopman operator},'' \emph{Journal of Nonlinear
  Science}, vol.~28, no.~2, pp. 687--710, 2018.

\bibitem{Zuogong_CDC}
Z.~Yue, J.~Thunberg, L.~Ljung, and J.~Goncalves, ``{Identification of Sparse
  Continuous-Time Linear Systems with Low Sampling Rate: Exploring Matrix
  Logarithms},'' https://arxiv.org/abs/1605.08590.

\bibitem{Brunton_ident_input}
S.~L. Brunton, J.~L. Proctor, and J.~N. Kutz, ``{Sparse identification of
  nonlinear dynamics with control (SINDYc)},'' in \emph{Proceedings of the IFAC
  Conference}, vol.~49, no.~18, 2016, pp. 710--715.

\bibitem{Proctor_input}
L.~P. Proctor, S.~L. Brunton, and J.~N. Kutz, ``{Generalizing Koopman operator
  theory to allow for inputs and control},'' \emph{SIAM Journal on Applied
  Dynamical Systems}, vol.~17, no.~1, pp. 909--930, 2018.

\bibitem{Proctor_DMD}
J.~L. Proctor, S.~L. Brunton, and J.~N. Kutz, ``Dynamic mode decomposition with
  control,'' \emph{SIAM Journal on Applied Dynamical Systems}, vol.~15, no.~1,
  pp. 142--161, 2016.

\bibitem{Shnitzer}
T.~Shnitzer, R.~Talmon, and J.-J. Slotine, ``Manifold learning with contracting
  observers for data-driven time-series analysis,'' \emph{IEEE Transactions on
  Signal Processing}, vol.~65, no.~4, pp. 904--918, 2016.

\bibitem{Perrault}
D.~Perrault-Joncas and M.~Meila, ``Estimating vector fields on manifolds and
  the embedding of directed graphs,'' \emph{arXiv preprint arXiv:1406.0013},
  2014.

\bibitem{Lasso}
R.~Tibshirani, ``Regression shrinkage and selection via the lasso,''
  \emph{Journal of the Royal Statistical Society. Series B (Methodological)},
  pp. 267--288, 1996.

\bibitem{radial_functions_dense_Sobolev}
G.~Ferrari-Trecate and R.~Rovatti, ``{Fuzzy systems with overlapping Gaussian
  concepts: Approximation properties in Sobolev norms},'' \emph{Fuzzy Sets and
  Systems}, vol. 130, no.~2, pp. 137--145, 2002.

\bibitem{Gardner}
T.~Gardner, C.~R. Cantor, and J.~J. Collins, ``{Construction of a genetic
  toggle switch in Escherichia coli},'' \emph{Nature}, vol. 403, pp. 339--342,
  2000.

\bibitem{Yall1}
J.~Yang and Y.~Zhang, ``{Alternating direction algorithms for L1-problems in
  compressive sensing},'' \emph{SIAM Journal on Scientific Computing}, vol.~33,
  pp. 250--278, 2011.

\bibitem{Yall1_manual}
Y.~Zhang, J.~Yang, and Y.~W., \emph{YALL1: Your ALgorithms for L1},
  yall1.blogs.rice.edu, 2011.

\bibitem{Kevrekidis_learning}
Q.~Li, F.~Dietrich, E.~M. Bollt, and I.~G. Kevrekidis, ``{Extended dynamic mode
  decomposition with dictionary learning: a data-driven adaptive spectral
  decomposition of the Koopman operator},'' \emph{Chaos: An Interdisciplinary
  Journal of Nonlinear Science}, vol.~27, p. 103111, 2017.

\bibitem{Wei_Kalman}
W.~Pan, F.~Menolascina, and G.-B. Stan, ``Online model selection for synthetic
  gene networks,'' in \emph{Proceedings of the 55th IEEE Conference on Decision
  and Control}.\hskip 1em plus 0.5em minus 0.4em\relax IEEE, 2016, pp.
  776--782.

\bibitem{Brunton_network_inference}
N.~M. Mangan, S.~L. Brunton, J.~L. Proctor, and J.~N. Kutz, ``Inferring
  biological networks by sparse identification of nonlinear dynamics,''
  \emph{IEEE Transactions on Molecular, Biological and Multi-Scale
  Communications}, vol.~2, no.~1, pp. 52--63, 2016.

\bibitem{subspace_identif}
P.~Van~Overschee and B.~De~Moor, \emph{{Subspace identification for linear
  systems: Theory—Implementation—Applications}}.\hskip 1em plus 0.5em minus
  0.4em\relax Springer Science \& Business Media, 2012.

\bibitem{modulating_functions}
M.~Shinbrot, ``On the analysis of linear and nonlinear systems,'' \emph{Trans.
  ASME}, vol.~79, no.~3, pp. 547--552, 1957.

\end{thebibliography}


\end{document}